\documentclass[11pt,reqno]{amsart}
\usepackage{tikz-cd}
\usepackage[T1]{fontenc}
\usepackage{amsmath,amssymb,amsthm,mathtools}
\usepackage[margin=1.05in]{geometry}
\usepackage{enumitem}
\usepackage[colorlinks=true,linkcolor=blue,citecolor=blue,urlcolor=blue]{hyperref}
\usepackage[all]{xy}
\usepackage{booktabs}
\DeclareMathOperator{\proj}{proj}
\theoremstyle{plain}
\newtheorem{theorem}{Theorem}[section]
\newtheorem{proposition}[theorem]{Proposition}
\newtheorem{lemma}[theorem]{Lemma}
\newtheorem{corollary}[theorem]{Corollary}

\theoremstyle{definition}
\newtheorem{definition}[theorem]{Definition}
\newtheorem{remark}[theorem]{Remark}
\newtheorem{example}[theorem]{Example}
\newtheorem{question}[theorem]{Question}
\newcommand{\K}{K}
\newcommand{\Z}{\mathbb{Z}}
\newcommand{\N}{\mathbb{N}}
\newcommand{\Gr}{\mathrm{Gr}}
\newcommand{\qgr}{\mathrm{qgr}}
\newcommand{\gr}{\mathrm{gr}}

\newcommand{\IBN}{\mathrm{IBN}}

\newcommand{\LE}{L_\K(E)}
\newcommand{\LRn}{L_\K(R_n)}

\DeclareMathOperator{\Tors}{Tors}

\begin{document}

\title{When is the twist of a Leavitt path algebra again a Leavitt path algebra?}

\author{Tran Giang Nam}
\address{Institute of Mathematics, Vietnam Academy of Science and Technology, 18 Hoang Quoc Viet, Hanoi, Vietnam}
\email{tgnam@math.ac.vn}

\author{Ashish K. Srivastava}
\address{Department of Mathematics and Statistics, Saint Louis University, St. Louis,
MO-63103, USA}
\email{ashish.srivastava@slu.edu}

\begin{abstract}
In this paper, we study Zhang twist of Leavitt path algebra $L_K(E)$ using a graded automorphism $\sigma$ induced by a graph automorphism such that the twisted algebra is a Leavitt path algebra over another graph $E_{\sigma}$ which is a twisted graph obtained from the original graph $E$. We also show that the hypothesis that $\sigma$ be induced by a graph automorphism cannot be dropped as we construct a graded automorphism of the Leavitt path algebra of an infinite graph whose Zhang twist is not isomorphic to any Leavitt path algebra. We also define the notion of noncommutative projective scheme for $\mathbb Z$-graded algebras that coincides with the notion of noncommutative projective scheme for connected $\mathbb N$-graded algebras defined by Artin and Zhang and study it in the context of twists of Leavitt path algebras. We establish combinatorial connection between the graphs $E$ and $E_{\sigma}$ and as a consequence, we show that many ring-theoretic properties are invariant for Leavitt path algebras under the twist by $\sigma$. In particular, we show that the invariant basis number property is preserved under Zhang twists considered here for finite graphs, whereas the module type is not preserved.

\bigskip	
\textbf{Mathematics Subject Classifications 2020}: 14A22, 16D90, 16S88, 16W50,  05C25\medskip

\textbf{Key words}: Leavitt path algebra; Zhang twist;  noncommutative projective scheme.
\end{abstract}

\maketitle

\section{Introduction}\label{sec:intro}

\noindent Twisting the multiplication of a graded algebra is one of the standard tools for producing new algebras with prescribed homological behavior. A notion of twist of a graded algebra $A$ was introduced by Artin, Tate, and Van den Bergh in \cite{ATV} as a deformation of the original graded product of $A$ with the help of a graded automorphism of $A$. Let $\sigma$ be an automorphism of the graded algebra $A=\oplus A_n$. Define a new multiplication $\star$ on the underlying graded $K$-module $\oplus A_n$ by $a \star b=a\sigma^n(b)$ where $a$ and $b$ are homogeneous elements in $A=\oplus A_n$ and $deg(a)=n$. The new graded algebra with the same underlying graded $K$-module $\oplus A_n$ and the new graded product $\star$ is called the twist of $A$ and is denoted as $A^{\sigma}$. 

This notion of twist of a graded algebra was later generalized by Zhang in \cite{Zhang96}, where he introduced the concept of twisting of graded product with the help of a twisting system. Let $\tau=\{\tau_n\mid n \in \mathbb Z\}$ be a set of graded $K$-linear automorphisms of $A=\oplus A_n$. Then $\tau$ is called a {\it twisting system} if $\tau_n(y\tau_m(z))=\tau_n(y)\tau_{n+m}(z)$ for all $n, m, l\in \mathbb Z$ and $y\in A_m$, $z\in A_l$. For example, if $\sigma$ is a graded algebra automorphism of $A$, then $\tau=\{\sigma^n\mid n\in \mathbb Z\}$ is a twisting system. Thus, the twist of a graded algebra in the sense of Artin-Tate-Van den Bergh can be viewed as a special case of the twist introduced by Zhang. Such a twist of a graded algebra $A$ is now known as Zhang twist and denoted as $A^\tau$. 

Zhang twist of a graded algebra has played an important role in the interaction of noncommutative algebra with noncommutative algebraic geometry. The central idea behind the noncommutative projective scheme is to generalize the category of quasi-coherent sheaves to the noncommutative case. In the case of commutative algebras, Serre's theorem established that studying the category of quasi-coherent sheaves on a projective variety is essentially the same as studying the quotient category of graded modules. This led Artin and Zhang to define the notion of noncommutative projective space. 

Zhang proved the fundamental result that for any algebra $A$, the graded module categories $\Gr$-$A$ and $\Gr$-$A^\tau$ are equivalent. If the algebra $A$ is connected $\mathbb N$-graded and noetherian, then this equivalence restricts to the subcategories of finitely generated modules to give an equivalence $\gr$-$A \cong \gr$-$A^\tau$. Moreover, the subcategories of modules which are torsion also correspond, consequently, we have an equivalence between the quotient categories $\qgr$-$A$ and $\qgr$-$A^\tau$. As a consequence it follows that their noncommutative projective schemes $\proj$-$A$ and $\proj$-$A^\tau$ are equivalent. Since Zhang twist of a commutative graded algebra by a non-identity automorphism yields a noncommutative graded algebra, this gives us a tool to construct examples of noncommutative graded algebras whose noncommutative projective schemes are isomorphic to commutative projective schemes. It is known that many fundamental properties like Gelfand-Kirillov dimesnion and Artin-Schelter regularity are preserved under Zhang twist whereas some ring-theoretic properties such as being a prime ring or being a PI ring are not preserved under Zhang twist. The real motivation behind this paper is the following question: 

\begin{question} \label{q:IBN}
If a $\Z$-graded algebra $A$ has the $\IBN$ property, does its Zhang twist $A^\tau$ also have the $\IBN$ property?
\end{question}

Recall that a unital ring $R$ has the \emph{invariant basis number} property ($\IBN$) if $R^m\cong R^n$ as right $R$-modules implies $m=n$. If $R$ does not have the $\IBN$ property, then the pair $(m,n)$ with $m<n$ and $R^m \cong R^n$, where $m$ is minimal and, subject to this, $n$ is minimal, is called the \emph{module type} of $R$. Note that if $A$ is an $\mathbb{N}$-graded algebra, then it is not difficult to see that $A$ has $\IBN$ property if and only if its Zhang twist $A^\tau$ has $\IBN$ property. This follows from a simple observation that an $\mathbb N$-graded algebra $A$ has $\IBN$ property if and only if its degree-zero component $A_0$ has $\IBN$ property. However, for $\Z$-graded algebras, Question \ref{q:IBN} turns out to be surprisingly difficult. The search for counter-examples naturally leads us to the world of Leavitt path algebras. Leavitt \cite{Leavitt62} produced, for each $n\ge2$, a universal example of a non-$\IBN$ algebra $L_K(1, n)$ of module type $(1,n)$ in the sense that $L_K(1,n)\cong L_K(1, n)^n$ as right $L_K(1, n)$-modules. This algebra is the Leavitt path algebra $\LRn$ of the rose graph with $n$ petals. This suggests that one possible way to come up with a counter-example would be to start with a Leavitt path algebra $L_K(E)$ having the $\IBN$ property and then twist with a graded automorphism in such a way that the resulting twisted algebra is a Leavitt path algebra without the $\IBN$ property. 
We note that criteria for the Leavitt path algebra of a finite graph to have the $\IBN$ property have  been established in \cite{KO:IBN, NP}. This leads naturally to the following question:

\begin{question} \label{q:LPA}
When can the Zhang twist of a Leavitt path algebra again be realized as a Leavitt path algebra?
\end{question} 

In \cite{NSV}, a construction of special graded automorphisms $\varphi$ which fix vertices was given for any Leavitt path algebra $L_K(E)$ and it was shown that for any graph $E$ and any such a graded automorphism $\varphi$, $L_K(E)$ embeds in the twisted algebra $L_K(E)^{\varphi}$. In this paper, we show, quite surprisingly, that this remains true for every graded automorphism of $L_K(E)$ that fixes the vertices, and we construct a family of graded automorphisms that do not fix the vertices. More precisely, we show that if $\sigma$ is a graph automorphism of $E$, then it induces a natural graded automorphism of the Leavitt path algebra $L_K(E)$, which we also denote by $\sigma$ and the Zhang twist of $L_K(E)$ by this graded automorphism $\sigma$ is isomorphic to the Leavitt path algebra $L_K(E_{\sigma})$ associated to another graph $E_{\sigma}$, which is obtained from the original graph $E$ by twisting its edges according to $\sigma$ (Theorem \ref{thm:twist-is-LPA}). 

We show that Question \ref{q:IBN} has an affirmative answer whenever the twist is given by a graded automorphism of finite order (Theorem~\ref{thm:IBN-finite-order}). In particular, for a finite graph $E$ and any automorphism $\sigma$ of $E$, the algebra $\LE$ has $\IBN$ if and only if $\LE^\sigma$ does (Corollary~\ref{cor:IBN-graphs}). By contrast, the module type need not be preserved: we exhibit a finite graph $E$ for which $\LE$ has module type $(1,2)$ while $\LE^\sigma$ has module type $(1,3)$ (Example~\ref{ex:type-changes}). Question \ref{q:IBN} thus remains open only for twists by automorphisms of infinite order and for general twisting systems. In the former case, a counterexample would require the skew Laurent polynomial ring $A[t,t^{-1};\sigma]$ to fail $\IBN$ (Remark \ref{rem:order}).

We also show that the hypothesis that the graded automorphism of $\LE$ is induced by a graph automorphism is crucial in showing that the twisted algebra also admits a Leavitt path algebra structure. We show that the twist of a Leavitt path algebra by a graded automorphism that is not induced by a graph automorphism need not be a Leavitt path algebra. For the graph $E$ consisting of a left-infinite tail attached to the root of the infinite binary tree, we construct a graded automorphism $\sigma$ of $\LE$ such that $\LE^\sigma\cong\bigoplus_{\Z}C[t,t^{-1}]$, where $C$ is the algebra of locally constant $K$-valued functions on the Cantor set (Example~\ref{ex:twist-not-LPA}). The algebra $\LE^\sigma$ is commutative but is not a Leavitt path algebra. Thus Theorem \ref{thm:twist-is-LPA} along with Example \ref{ex:twist-not-LPA} gives a partial answer to Question \ref{q:LPA}.

It turns out that there is a strong combinatorial connection between the graphs $E$ and $E_{\sigma}$. As a consequence, we show that many ring-theoretic properties for unital Leavitt path algebras are invariant under twisting by such a graded automorphism $\sigma$. In particular, we show that the properties of being a von Neumann regular ring, exchange ring, Dedekind finite, stably finite and strongly graded are invariant under twisting by $\sigma$ for unital Leavitt path algebras. We also construct an example in which $L_K(E)$ is noncommutative, whereas the twisted algebra $L_K(E)^{\sigma}$ is a commutative Leavitt path algebra. 

As remarked earlier, in the case of a connected $\mathbb N$-graded algebra $A$, the noncommutative projective schemes $\proj$-$A$ and $\proj$-$A^\tau$ are equivalent. To study $\proj$ in the context of Leavitt path algebras, we first extend the notion of a noncommutative projective scheme to $\mathbb Z$-graded algebras. For a $\mathbb Z$-graded algebra $A=\bigoplus_{n\in\Z}A_n$, we define the noncommutative projective scheme of $A$ as the triple $\proj A =\; \big(\mathsf{QGr}_I\text{-}A,\ \pi(A_A),\ s\big)$, where $\mathsf{QGr}_I\text{-}A =\; \Gr\text{-}A\big/\Tors_I(A)$, with $\Tors_{I}(A)$ denoting the smallest localizing subcategory of $\Gr\text{-}A$ containing every graded module $M$ such that $MI_A=0$, where $I_A =\; A\Big(\bigoplus_{n\neq 0}A_n\Big)A$ and $s$ is the autoequivalence induced by the degree shift $M\mapsto M(1)$. Our notion of a noncommutative projective scheme for a $\mathbb Z$-graded algebra coincides with the notion of a noncommutative projective scheme for a connected $\mathbb N$-graded algebra defined by Artin and Zhang \cite{AZ}. We then study $\proj$ in the context of twists of Leavitt path algebras. For the graph $E=C_m$ with vertices $v_1,\dots,v_m$ and edges $e_i\colon v_i\to v_{i+1}$, where the indices are taken modulo $m$, $\sigma$ being the rotation $v_i\mapsto v_{i+1}$, $e_i\mapsto e_{i+1}$, we show that $\proj \LE$ and $\proj \LE^\sigma$ have the same underlying pair $(\mathcal C,\mathcal O)$, but they differ in the shift functor (Example \ref{ex:Cm}).

The article is organized as follows. In Section \ref{LPAand Twist}, we first show that the Leavitt path algebra $L_K(E)$ of an arbitrary graph $E$ over a field $K$ embeds into the twisted $K$-algebra $L_K(E)^{\varphi}$, where $\varphi$ is a graded automorphism of $L_K(E)$ that fixes every vertex (Proposition \ref{prop:embed}). We then construct a family of graded automorphisms $\sigma$ of $L_K(E)$ induced by graph automorphisms of $E$ that need not fix the vertices (Lemma \ref{lem:graphaut}); we use the same notation $\sigma$ for these induced automorphisms. We next show that the Zhang twist of $L_K(E)$ by a graded automorphism induced by an automorphism $\sigma$ of $E$ is isomorphic to the Leavitt path algebra $L_K(E_{\sigma})$ of a graph $E_{\sigma}$ obtained from $E$ by twisting its edges according to $\sigma$ (Theorem \ref{thm:twist-is-LPA}). We give an example in which $L_K(E)$ is noncommutative, whereas the twisted algebra $L_K(E)^{\sigma}$ is commutative (Example \ref{ex:C2}). We also show that the twist of a Leavitt path algebra by a graded automorphism that does not permute the vertices need not be a Leavitt path algebra (Example~\ref{ex:twist-not-LPA}).

In Section \ref{noncom-Proj}, we introduce the noncommutative projective scheme associated to a unital $\mathbb{Z}$-graded algebra $A$, extending the notion of Artin and Zhang \cite{AZ} for connected $\mathbb{N}$-graded algebras (Definition \ref{def:proj} and Proposition \ref{prop:AZ-recovered}). We then describe the noncommutative projective scheme associated to a strongly $\mathbb{Z}$-graded algebra (Proposition \ref{prop:strong}) and show that this notion of $\proj$ for $\mathbb{Z}$-graded algebras is well suited to Zhang twists by graded automorphisms (Theorem \ref{thm:proj-twist} and Proposition \ref{prop:degree-one-unit}). As a consequence, we describe the noncommutative projective scheme associated to the Zhang twist of the Leavitt path algebra $L_K(E)$ of a finite graph without sinks by a graded automorphism induced by an automorphism $\sigma$ of $E$ (Corollary \ref{cor:proj-LPA-twist}). Finally, by combining these results with the work of Smith \cite{S}, we obtain a categorical equivalence between $\mathsf{Gr}\text{-}L_K(E)$ and $\mathsf{Gr}\text{-}L_K(E)^{\sigma}$ (Corollary \ref{cor:smith-dictionary}), without appealing to Zhang's theorem \cite{Zhang96}.

In Section \ref{Invariant}, we establish a combinatorial connection between the graphs $E$ and $E_\sigma$ (Corollary \ref{cor:acyclic}). As a consequence, we show that many ring-theoretic properties are invariant for unital Leavitt path algebras under the twist by $\sigma$, including strong gradedness, von Neumann regularity, Dedekind finiteness, stable finiteness, the exchange property, the UGN property, and the gr-$\IBN$ property (Propositions \ref{prop:strongly-graded-preserved}, and \ref{prop:UGN}). We also show that the $\IBN$ and UGN properties are invariant under the twist of an arbitrary unital $\Z$-graded algebra by a graded automorphism of finite order (Theorem~\ref{thm:IBN-finite-order}), and that the module type is not invariant (Example~\ref{ex:type-changes}).

\section{Leavitt path algebras and their twists}\label{LPAand Twist}

\noindent A (directed) graph $E = (E^0, E^1, s, r)$ consists of two disjoint sets $E^0$ and $E^1$, called \emph{vertices} and \emph{edges} respectively, together with two maps $s, r: E^1 \longrightarrow E^0$.  The vertices $s(e)$ and $r(e)$ are referred to as the \emph{source} and the \emph{range} of the edge~$e$, respectively.  A graph $E$ is called {\it row-finite} if $s^{-1}(v)$ is finite for all $v\in E^0$. It is called {\it finite} if both $E^0$ and $E^1$ are finite.

A {\em sink} in a graph $E$ is a vertex $v \in E^0$ with $s^{-1}(v) = \emptyset$; a {\em source} is a vertex $v \in E^0$ with $r^{-1} (v) =\emptyset$. A {\it regular vertex} in a graph $E$ is a vertex $v$ with $0 < |s^{-1}(v)| < \infty$.

A (finite) \emph{path} in a graph $E$ is a string
$p=e_1\cdots e_n$ of edges $e_i\in E^1$ such that $r(e_i) = s(e_{i+1})$ for all $i$. The \emph{length} of the path $p = e_1 \cdots e_n$ is $n$, and is denoted by $|p|$.  
The source and range maps on edges are extended to paths as
\[s(p)=s(e_1)\qquad\text{and}\qquad r(p)=r(e_n).\]
Vertices are  regarded as paths of length $0$, with each vertex coinciding with its source and its range. We denote by $\text{Path}(E)$ the set of all paths in $E$.

An edge $f$ is called an {\it exit for a path $p = e_1\cdots e_n$} if there exists $i$ ($1\le i\le n$) such that $s(e_i) = s(f)$ and $e_i\neq f$.
A path $p$ of positive length is called a {\it closed path based at} $v$ if $s(p) = r(p) = v$. A {\it closed simple path based at $v$} is a closed path $p=e_1 \cdots e_n$ based at $v$ such that $v \neq s(e_i)$ for all $1 < i \le n$. A {\it cycle based at $v$} is a closed path $p=e_1 \cdots e_n$ based at $v$ and $s(e_i)\neq s(e_j)$ for all $i\neq j$. A graph is called {\it acyclic} if  it does not have any cycles.

The Leavitt path algebra of a graph with coefficients
in a field  was introduced by Abrams and Aranda Pino in \cite{AP05}, and independently by Ara, Moreno and Pardo in \cite{AMP}. Leavitt path algebras are algebraic analogues of graph $C^*$-algebras and provide natural generalizations of the Leavitt algebras of type $(1,n)$ introduced by William Leavitt in \cite{Leavitt62}.

\begin{definition}[{\cite[Definition 1.2.3]{AAS}}]\label{LPAs}
For a graph $E = (E^0,E^1,s,r)$ and any  field $K$, the \emph{Leavitt path algebra} $L_{K}(E)$ {\it of the graph}~$E$
\emph{with coefficients in}~$K$ is the $K$-algebra generated
by the union of the set $E^0$  and two disjoint copies $E^1$, say $E^1$ and $\{e^*\mid e\in E^1\}$, satisfying the following relations for all $v, w\in E^0$ and $e, f\in E^1$:
	\begin{itemize}
		\item[(1)] $v w = \delta_{v, w} w$;
		\item[(2)] $s(e) e = e = e r(e)$ and $e^*s(e) = e^* = r(e)e^*$;
		\item[(3)] $e^* f = \delta_{e, f} r(e)$;
		\item[(4)] $v= \sum_{e\in s^{-1}(v)}ee^*$ for any  regular vertex $v$;
	\end{itemize}
	where $\delta$ is the Kronecker delta.
\end{definition}

The Leavitt path algebra $\LE$ has the following universal property: if $A$ is any $K$-algebra generated by a family of elements $\{a_v, b_e, c_{e^{\ast}}\mid v\in E^0, e\in E^1\}$ satisfying the relations analogous to (1) - (4) in Definition \ref{LPAs}, then there exists a unique $K$-algebra homomorphism $\varphi: \LE \longrightarrow A$ given by 
\begin{center}
$\varphi(v)=a_v$, $\varphi(e)=b_e$, and $\varphi(e^{\ast})=c_{e^{\ast}}$. 
\end{center}
It can be shown that $\LE$ is spanned as a $K$-vector space by $\{pq^{\ast}\mid p, q \in \text{Path}(E), r(p)=q(y)\}$ where if $q = v \in E^0$, we set $q^*=v$, and if  $q=e_1\cdots e_m$, with $e_i\in E^1$, then $q^*=e_m^*\cdots e_1^*$. 
Every Leavitt path algebra $\LE$ is a $\mathbb{Z} $\textit{-graded algebra}, namely, $\LE={\displaystyle\bigoplus\limits_{n\in\mathbb{Z}}}
L_K(E)_{n}$ induced by defining, for all $v\in E^{0}$ and $e\in E^{1}$, $\deg
(v)=0$, $\deg(e)=1$, $\deg(e^{\ast})=-1$. For each $n\in\mathbb{Z}$, the \textit{homogeneous component }$L_K(E)_{n}$
is given by
\[
L_K(E)_{n}=\text{span}_K\{pq^{\ast}\mid p, q \in \text{Path}(E),
|p|-|q|=n\}.
\]

Let $K$ be a field, $E$ a graph, and $\phi$ a graded algebra automorphism of $L_K(E)$. We then have that $\tau=\{\varphi^m\}_{m\in\Z}$ is a normalized twisting system on $L_K(E)$, and the twist $L_K(E)^\varphi$ carries
\[a\star b=a\,\varphi^m(b)\] for all $a\in L_K(E)_m$ and $b\in L_K(E)$. In particular, if $L_K(E)$ is unital with identity $1$, then $L_K(E)^\varphi$ is also unital with identity $1$, since 
\begin{center}
$1\star b=b$ and $a\star1=a\varphi^m(1)=a$.    
\end{center}

A family of graded automorphisms of $L_K(E)$ fixing all vertices was introduced in \cite[Theorem 2.2]{KN2023} and \cite[Theorem 2.2 and Corollary 2.3]{NSV}. For such a graded automorphim $\varphi$, \cite[Proposition 3.2]{NSV} showed that $L_K(E)$ can be embedded into $L_K(E)^{\varphi}$. The following proposition extends this result to arbitrary graded automorphisms fixing all vertices.

\begin{proposition}\label{prop:embed}
Let $E$ be a graph, $K$ a field, and $\varphi$ a graded automorphism of $L_K(E)$ fixing all vertices. Then there exists a graded injective homomorphism $\Psi\colon \LE\longrightarrow \LE^{\varphi}$ of $K$-algebras satisfying
\[
\Psi(v)=v\quad(v\in E^0),\qquad \Psi(e)=e,\qquad \Psi(e^*)=\varphi^{-1}(e^*)\quad(e\in E^1).
\] If, in addition, $E^0$ is finite, then $\Psi$ is a unital $K$-algebra homomorphism.
\end{proposition}

\begin{proof}
We first note that $\varphi^{-1}$ is a graded $K$-algebra automorphism of $L_K(E)$ fixing all vertices, and so $\varphi^{-1}(v) =v$ for all $v\in E^0$ and $\deg(\varphi^{-1}(e^*))=-1$ for all $e\in E^1$.

We define the elements $\{Q_v\mid v\in E^0\}$ and $\{T_e, T_{e^*}\mid e\in E^1\}$ of $L_K(E)^{\varphi}$ by setting 
\begin{center}
$Q_v = v$,\quad $T_e = e$,\quad and\quad  $T_{e^*} = \varphi^{-1}(e^*)$.    
\end{center}
We obtain that 
\begin{center}
$\deg(Q_v) = 0$,\quad $\deg(T_e) = 1$,\quad and\quad $\deg(T_{e^*}) = -1$    
\end{center}
for all $v\in E^0$ and $e\in E^1$.

We claim that $\{Q_v, T_e, T_{e^*}\mid v\in E^0, e\in E^1\}$ is a family in $L_K(E)^{\varphi}$ satisfying the relations analogous to (1) - (4) in Definition \ref{LPAs}. Indeed, we have 
$$Q_v\star Q_w= v\star w = v\varphi^0(w)=vw=\delta_{v,w}w = \delta_{v,w}Q_w,$$
for all $v, w\in E^0$, showing relation $(1)$.

For $(2)$, we have
 $$Q_{s(e)}\star T_e=s(e)\star e= s(e)\varphi^0(e) = s(e)e = e = T_e,$$ 
 
 \[T_e\star Q_{r(e)}=e\star r(e) =e\varphi(r(e))=er(e) = e = T_e,\]

\[T_{e^*}\star Q_{s(e)}=\varphi^{-1}(e^*)\,\varphi^{-1}\bigl(s(e)\bigr)=\varphi^{-1}\bigl(e^*s(e)\bigr)=\varphi^{-1}(e^*)= T_{e^*},
\]
and

\[Q_{r(e)}\star T_{e^*}=r(e)\,\varphi^{-1}(e^*)= \varphi^{-1}(r(e))\,\varphi^{-1}(e^*)=\varphi^{-1}\bigl(r(e)e^*\bigr)=\varphi^{-1}(e^*)=T_{e^*}
\] for all $e\in E^1$.

For $(3)$, we obtain that
\[T_{e^*}\star T_f=  \varphi^{-1}(e^*)\,\varphi^{-1}(f)=\varphi^{-1}(e^*f)=\varphi^{-1}\bigl(\delta_{e,f}r(e)\bigr)=\delta_{e,f}\,r(e) = \delta_{e,f}Q_{r(e)}\] for all $e, f\in E^1$.

For $(4)$, let $v$ be a regular vertex in $E$. We then have
\[\sum_{e\in s^{-1}(v)} T_e\star T_{e^*}=\sum_{e\in s^{-1}(v)}e\,\varphi\bigl(\varphi^{-1}(e^*)\bigr)=\sum_{e\in s^{-1}(v)}e\,e^*=v= Q_v,
\] thus showing the claim. Then, by the Universal Property of $L_K(E)$, there is a unique $\K$-algebra homomorphism $\Psi\colon \LE\longrightarrow \LE^{\varphi}$ such that 
\begin{center}
$\Psi(v) = Q_v$,\quad $\Psi(e)= T_e$,\quad and\quad $\Psi(e^*) = T_{e^*}$    
\end{center}
for all $v\in E^0$ and $e\in E^1$. Since 
$\deg(Q_v) = 0$, $\deg(T_e) = 1$, and $\deg(T_{e^*}) = -1$    
for all $v\in E^0$ and $e\in E^1$, it follows that $\Psi$ is a $\mathbb{Z}$-graded homomorphism, and therefore the injectivity of $\Psi$ is guaranteed by the Graded Uniqueness Theorem (see, e.g., \cite[Theorem 2.2.15]{AAS}).

In particular, if $E^0$ is finite, then $1_{L_K(E)} = \sum_{v\in E^0}v$ (see, e.g., \cite[Lemma 1.2.12(iv)]{AAS}), and so $$\Psi(1_{L_K(E)})=\sum_{v\in E^0}v=1_{L_K(E)^{\varphi}},$$ showing that $\Psi$ is a unital $K$-algebra homomorphism, thus finishing the proof.
\end{proof}

In \cite[Section 3]{NSV}, the authors and Vien provided  special examples in which the homomorphism $\Psi$, introduced in Proposition \ref{prop:embed},  is an isomorphism. This naturally raises the question of when the twist  $L_K(E)^{\varphi}$ of a Leavitt path algebra $L_K(E)$ by a graded automorphism $\varphi$ is itself a Leavitt path algebra.
The next part of this section provides a partial answer to this question by considering a family of graded automorphisms of  $L_K(E)$ arising from automorphisms of $E$.

For the reader's convenience, we recall the notion of graph automorphisms.

\begin{definition}
For every graph $E$, an \emph{automorphism of} $E$ is a pair $\sigma=(\sigma_0,\sigma_1)$ of bijections $\sigma_0\colon E^0\to E^0$ and $\sigma_1\colon E^1\to E^1$ with
\[
s\circ\sigma_1=\sigma_0\circ s,\qquad r\circ\sigma_1=\sigma_0\circ r .
\]
\end{definition}

The following lemma gives graded automorphisms of $L_K(E)$ arising from automorphisms of $E$.

\begin{lemma}\label{lem:graphaut}
Let $K$ be a field, $E$ a graph, and $\sigma=(\sigma_0,\sigma_1)$ an automorphism of $E$. Then
\[
\sigma(v)=\sigma_0(v)\quad (v\in E^0) ,\qquad \sigma(e)=\sigma_1(e),\qquad \sigma(e^*)=\sigma_1(e)^*\quad (e\in E^1)
\]
extends to a graded $\K$-algebra automorphism of $\LE$. In particular, if $E^0$ is finite, then the automorphism is unital.
\end{lemma}
\begin{proof}
We define the elements $\{Q_v\mid v\in E^0\}$ and $\{T_e, T_{e^*}\mid e\in E^1\}$ of $L_K(E)$ by setting 
\begin{center}
$Q_v = \sigma_0(v)$,\quad $T_e = \sigma_1(e)$,\quad and\quad $T_{e^*} = \sigma_1(e)^*$.    
\end{center}
We claim that $\{Q_v, T_e, T_{e^*}\mid v\in E^0, e\in E^1\}$ is a family in $L_K(E)$ satisfying the relations analogous to (1) - (4) in Definition \ref{LPAs}. Indeed, for all $v, w\in E^0$, we have 
$$Q_vQ_w=\sigma_0(v)\sigma_0(w)=\delta_{\sigma_0(v),\sigma_0(w)}\sigma_0(w)=\delta_{v,w}\sigma_0(w) = \delta_{v,w} Q_w,$$ where the third equality follows from the injectivity of $\sigma_0$. Therefore, relation $(1)$ holds. 

For $(2)$, for each $e\in E^1$, we have 
$$Q_{s(e)}T_e=\sigma_0(s(e))\sigma_1(e)=s(\sigma_1(e))\sigma_1(e)=\sigma_1(e)= T_e,$$
$$T_eQ_{r(e)} = \sigma_1(e)\sigma_0(r(e))=  \sigma_1(e)r(\sigma_1(e))=\sigma_1(e)= T_e,$$
$$Q_{r(e)}T_{e^*} = \sigma_0(r(e)) \sigma_1(e)^* = r(\sigma_1(e)) \sigma_1(e)^* = \sigma_1(e)^* = T_{e^*},$$
and 
$$T_{e^*}Q_{s(e)} = \sigma_1(e)^* \sigma_0(s(e)) = \sigma_1(e)^* s(\sigma_1(e)) = \sigma_1(e)^* = T_{e^*}.$$

For $(3)$, we obtain that
\[T_{e^*} T_f =\sigma_1(e)^*\sigma_1(f)=\delta_{\sigma_1(e),\sigma_1(f)}\,r(\sigma_1(e))=\delta_{e,f}\,\sigma_0(r(e))=\delta_{e,f} Q_{r(e)}
\] for all $e, f\in E^1$, where the third equality follows from the injectivity of $\sigma_1$. 

For $(4)$, let $v$ be a regular vertex in $E$. We note that $\sigma_1$ restricts to a bijection $s^{-1}(v)\to s^{-1}(\sigma_0(v))$, and so $\sigma_0(v)$ is also a regular vertex in $E$. We then have
\[
\sum_{e\in s^{-1}(v)} T_e T_{e^*}=\sum_{e\in s^{-1}(v)}\sigma_1(e)\sigma_1(e)^*=\sum_{f\in s^{-1}(\sigma_0(v))}ff^*=\sigma_0(v)= Q_v,
\]
thus showing the claim. By the Universal Property of $L_K(E)$,
there is a unique $\K$-algebra endomorphism of $L_K(E)$, also denoted by $\sigma$, such that 
\begin{center}
$\sigma(v) = Q_v$,\quad $\sigma(e)= T_e$,\quad and\quad $\sigma(e^*) = T_{e^*}$    
\end{center}
for all $v\in E^0$ and $e\in E^1$. Since 
$\deg(Q_v) = 0$, $\deg(T_e) = 1$, and $\deg(T_{e^*}) = -1$    
for all $v\in E^0$ and $e\in E^1$, it follows that $\sigma$ is a $\mathbb{Z}$-graded homomorphism.

Similarly, the graded $K$-algebra endomorphism of $L_K(E)$ induced by $\sigma^{-1}=(\sigma_0^{-1},\sigma_1^{-1})$ is a two-sided inverse of $\sigma$.

In particular, if $E^0$ is finite, then $1_{L_K(E)} = \sum_{v\in E^0}v$ (by \cite[Lemma 1.2.12(iv)]{AAS}), and hence
$$\sigma(1_{L_K(E)}) = \sigma(\sum_{v\in E^0}v) =\sum_{v\in E^0}\sigma_0(v) = \sum_{v\in E^0}v = 1_{L_K(E)},$$
thus finishing the proof.
\end{proof}

For clarification, we illustrate Lemma \ref{lem:graphaut}  by presenting the following example.

\begin{example} \label{C_m} \rm
Let $m$ be a positive integer, and let $C_m$ be the graph having vertices $v_1,\dots,v_m$ and edges $e_i\colon v_i\to v_{i+1}$, with indices taken modulo $m$. The rotation 
\begin{center}
$v_i\mapsto v_{i+1}$\quad and\quad $e_i\mapsto e_{i+1}$     
\end{center}
is a graph automorphism of order $m$, giving a graded automorphism of $L_\K(C_m)\cong M_m\bigl(\K[x,x^{-1}]\bigr)$ that cyclically permutes all vertices.
\end{example}

The twist by such an automorphism is again a Leavitt path algebra --- of a different graph which is a twisted graph constructed from the original graph.

\begin{definition} \label{def:twisted-quiver}
Let $E$ be a graph and $\sigma$ an automorphism of $E$. Define
\[
E_{\sigma}:=\bigl(E^0,\ E^1,\ s,\ r_{\sigma}\bigr),\qquad r_{\sigma}:=\sigma_0^{-1}\circ r=r\circ\sigma_1^{-1},
\]
where the two descriptions agree because $r\circ\sigma_1=\sigma_0\circ r$. Thus $E_{\sigma}$ has the same vertices and edges as $E$ and the same source map, while its range map is twisted by $\sigma$.
\end{definition}

For clarification, we illustrate Definition \ref{def:twisted-quiver}  by presenting the following example.

\begin{example}\label{ex:twist-quiver}
Consider the graph $E=C_2$: \[\begin{tikzcd}
	{\bullet v_1} && {\bullet v_2}
	\arrow["{e_1}", from=1-1, to=1-3]
	\arrow["{e_2}", shift left=3, from=1-3, to=1-1]
\end{tikzcd}\]
and let $\sigma$ be an automorphism of $E$ defined by 
\begin{center}
$\sigma(v_1)=v_2,\ \sigma(v_2)=v_1,\ \sigma(e_1)=e_2, \text{ and } \sigma(e_2)=e_1$.     
\end{center}
We then have 
\begin{center}
$r_{\sigma}(e_1)=\sigma_0^{-1}(v_2)=v_1=s(e_1)$ and $r_{\sigma}(e_2)=\sigma_0^{-1}(v_1)=v_2=s(e_2)$,     
\end{center}
and so $E_{\sigma} =R_1\sqcup R_1$:
\[\begin{tikzcd}
	{\bullet v_1} && {\bullet v_2}
	\arrow["{e_1}", from=1-1, to=1-1, loop, in=55, out=125, distance=10mm]
	\arrow["{e_2}", from=1-3, to=1-3, loop, in=55, out=125, distance=10mm]
\end{tikzcd}\]    
\end{example}

\begin{lemma}\label{lm:twsist-path}
Let $E$ be a graph, $\sigma$ an automorphism of $E$, and $k$ a positive integer. Then $e_1e_2\cdots e_k$ is a path in $E_{\sigma}$ if and only if $f_1f_2\cdots f_k$ is a path in $E$, where $f_i = \sigma_1^{i-1}(e_i)$ for all $1\le i\le k$.
\end{lemma}
\begin{proof}
We note that $e_1e_2\cdots e_k$ is a path in $E_{\sigma}$ if and only if $$r_\sigma(e_i)=s(e_{i+1})$$ for all $1\le i \le k-1$, if and only if $$\sigma_0^{-1}(r(e_i))=s(e_{i+1})$$ for all $1\le i \le k-1$, if and only if $$r(e_i)= \sigma_0(s(e_{i+1}))=s(\sigma_1(e_{i+1}))$$ for all $1\le i \le k-1$. Equivalently, applying $\sigma_0^{\,i-1}$ to the last equality, we obtain that
$$r(f_i)=s(f_{i+1})$$ for all $1\le i \le k-1$, i.e., $f_1f_2\cdots f_k$ is a path in $E$, thus finishing the proof.
\end{proof}

We are now in a position to state the main theorem of this section, which shows that the twist of a Leavitt path algebra $L_K(E)$ by a graded automorphism arising from an automorphism of $E$ is again a Leavitt path algebra.

\begin{theorem}\label{thm:twist-is-LPA}
Let $K$ be a field and  $E$ an arbitrary graph, and let $\sigma$ be an automorphism of $E$, with induced graded algebra automorphism of $L_K(E)$ also denoted by $\sigma$. Then
\[
L_\K(E)^{\sigma}\ \cong\ L_\K(E_{\sigma})
\]
as $\Z$-graded $\K$-algebras. In particular, if $E^0$ is finite, then the isomorphism is unital.
\end{theorem}

\begin{proof}
We define the elements $\{Q_v\mid v\in E^0\}$ and $\{A_e, B_{e^*}\mid e\in E^1\}$ of $L_K(E)^{\sigma}$ by setting 
\[
Q_v=v,\quad A_e=e,\quad \text{and}\quad  B_{e^*}=\sigma^{-1}(e^*)=\bigl(\sigma_1^{-1}(e)\bigr)^*.
\] We claim that $\{Q_v, A_e, B_{e^*}\mid v\in E^0, e\in E^1\}$ is a family in $L_K(E)^{\sigma}$ satisfying the relations analogous to (1) - (4) in Definition \ref{LPAs}. Indeed, for all $v, w\in E^0$, we have
$$Q_v\star Q_w= v \star w = v\,\sigma^0(w)=vw= \delta_{v,w} w=\delta_{v,w}Q_w,$$  which establishes relation $(1)$. 

For $(2)$, for each $e\in E^1$, we have
$$Q_{s(e)}\star A_e=s(e)\star e=s(e) \sigma^0(e)= s(e) e =e = A_e.$$
Since $\sigma(r_{\sigma}(e))=\sigma(\sigma_0^{-1}r(e))=r(e)$, we also have
\[
A_e\star Q_{r_{\sigma}(e)}= e\star r_\sigma(e)=e\,\sigma\bigl(r_{\sigma}(e)\bigr)=e\,r(e)=e = A_e.
\]
Since $\deg(B_{e^*})=-1$, we obtain that
\[
B_{e^*}\star Q_{s(e)}=\sigma^{-1}(e^*)\,\sigma^{-1}\bigl(s(e)\bigr)=\sigma^{-1}\bigl(e^*s(e)\bigr)=\sigma^{-1}(e^*) = B_{e^*},\]
and
\[Q_{r_\sigma(e)}\star B_{e^*}= \sigma^{-1}(r(e)) \sigma^{-1}(e^*)= \sigma^{-1}\bigl(r(e)e^*\bigr)= \sigma^{-1}(e^*)=B_{e^*}.
\]

For $(3)$, for all $e, f\in E^1$, we have
\[
B_{e^*}\star A_f=\sigma^{-1}(e^*)\,\sigma^{-1}(f)=\sigma^{-1}(e^*f)=\delta_{e,f}\,\sigma_0^{-1}\bigl(r(e)\bigr)=\delta_{e,f}\,Q_{r_{\sigma}(e)} .
\]

For $(4)$,  since $s_{\sigma}=s$, for a regular vertex $v$,
\[
\sum_{e\in s^{-1}(v)}A_e\star B_{e^*}=\sum_{e\in s^{-1}(v)}e\,\sigma\bigl(\sigma^{-1}(e^*)\bigr)=\sum_{e\in s^{-1}(v)}ee^*=v=Q_v,
\] thus proving the claim. By the Universal Property of $L_K(E_{\sigma})$, there is a unique $K$-algebra homomorphism $\Theta\colon L_\K(E_{\sigma})\longrightarrow L_K(E)^{\sigma}$ such that 
\begin{center}
$\Theta(v)=Q_v$,\quad $\Theta(e)=A_e$,\quad and \quad $\Theta(e^*)=B_{e^*}$   
\end{center}
for all $v\in E^0$ and $e\in E^1$. Since 
$\deg(Q_v) = 0$, $\deg(A_e) = 1$, and $\deg(B_{e^*}) = -1$    
for all $v\in E^0$ and $e\in E^1$, it follows that $\Theta$ is a $\mathbb{Z}$-graded homomorphism. Moreover, $\Theta(v)=v\ne0$ for every $v\in E^0$. Therefore, by the Graded Uniqueness Theorem (see, e.g., \cite[Theorem 2.2.15]{AAS}),
$\Theta$ is injective.

We next prove that $\Theta$ is surjective. We first claim that
\begin{equation}
A_{e_1}\star\cdots\star A_{e_k}\star B_{f^*_1}\star \cdots \star B_{f^*_m}=e_1\sigma_1(e_2)\cdots\sigma_1^{k-1}(e_k)\sigma^{k-1}_1(f_1)^*\cdots \sigma^{k-m}_1(f_m)^*
\end{equation} for all $e_i, f_j\in E^1$ and for all $k, m\in \mathbb{N}$ with $k + m\ge 1$.
We use induction on $k + m$ to establish the claim. If $k + m =1$, then either $(k, m) = (1, 0)$ or $(k, m) = (0, 1)$, in which case the claim is immediate. Now suppose that $k+ m\ge 2$, and assume inductively that the claim holds for all tuples of edges $(e_1, \ldots, e_k)$ and $(f_1, \ldots, f_m)$ satisfying $1 \le k + m < n$.  We prove the claim for $k + m = n$. If $k = m = 1$, then the claim follows immediately. If $m \ge 2$, then by the induction hypothesis, we have 
$$A_{e_1}\star\cdots\star A_{e_k}\star B_{f^*_1}\star \cdots \star B_{f^*_{m-1}}=e_1\sigma_1(e_2)\cdots\sigma_1^{k-1}(e_k)\sigma^{k-1}_1(f_1)^*\cdots \sigma^{k-m +1}_1(f_{m-1})^*,$$ and so 
\begin{align*}
A_{e_1}\star\cdots\star A_{e_k}\star B_{f^*_1}\star \cdots \star B_{f^*_m}  & = (A_{e_1}\star\cdots\star A_{e_k}\star B_{f^*_1}\star \cdots \star B_{f^*_{m-1}}) \star B_{f^*_m}\\
                       & = e_1\sigma_1(e_2)\cdots\sigma_1^{k-1}(e_k)\sigma^{k-1}_1(f_1)^*\cdots \sigma^{k-m +1}_1(f_{m-1})^*\star \sigma^{-1}_1(f_m)^* \\
                       & = e_1\sigma_1(e_2)\cdots\sigma_1^{k-1}(e_k)\sigma^{k-1}_1(f_1)^*\cdots \sigma^{k-m +1}_1(f_{m-1})^*\sigma^{k-m}_1(f_m)^*, \end{align*}
as desired.

If $k \ge 2$, then by the induction hypothesis, we have
$$A_{e_2}\star\cdots\star A_{e_k}\star B_{f^*_1}\star \cdots \star B_{f^*_{m}}=e_2\sigma_1(e_3)\cdots\sigma_1^{k-2}(e_k)\sigma^{k-2}_1(f_1)^*\cdots \sigma^{k-m -1}_1(f_{m})^*,$$
and hence
\begin{align*}
A_{e_1}\star\cdots\star A_{e_k}\star B_{f^*_1}\star \cdots \star B_{f^*_m}  & = A_{e_1}\star (A_{e_2}\cdots\star A_{e_k}\star B_{f^*_1}\star \cdots \star B_{f^*_{m-1}} \star B_{f^*_m})\\
                       & = e_1\star e_2\sigma_1(e_3)\cdots\sigma_1^{k-2}(e_k)\sigma^{k-2}_1(f_1)^*\cdots \sigma^{k-m -1}_1(f_{m})^*\\
                       & = e_1\sigma( e_2\sigma_1(e_3)\cdots\sigma_1^{k-2}(e_k)\sigma^{k-2}_1(f_1)^*\cdots \sigma^{k-m -1}_1(f_{m})^*)\\
                       &=e_1\sigma_1(e_2)\cdots\sigma_1^{k-1}(e_k)\sigma^{k-1}_1(f_1)^*\cdots \sigma^{k-m +1}_1(f_{m-1})^*\sigma^{k-m}_1(f_m)^*, \end{align*}
thus showing the claim.

We next prove that $\alpha\beta^*\in \text{Im}(\Theta)$ for all $\alpha, \beta\in \text{Path}(E)$ with $r(\alpha) = r(\beta)$. If $\alpha = v=\beta$ for some $v\in E^0$, then $\Theta(v) = Q_v = v$, as desired. If $|\alpha| \ge 1$ and $|\beta| =0$, then $\beta = r(\alpha)$. Write $\alpha = \alpha_1\cdots \alpha_k$, where $\alpha_i\in E^1$ for all $1\le i\le k$. Let $p := e_1\cdots e_k$, where $e_i = \sigma^{-(i-1)}_1(\alpha_i)$ for all $1\le i\le k$. By Lemma \ref{lm:twsist-path}, $p$ is a path in $E_{\sigma}$. Then, using formula $(1)$, we have
$$\Theta(p) = A_{e_1}\star A_{e_2}\star \cdots \star A_{e_k}=e_1\sigma_1(e_2)\cdots\sigma_1^{k-1}(e_k)= \alpha_1\alpha_2\cdots\alpha_k=\alpha,$$ and so $\alpha\in \text{Im}(\Theta)$, as desired.

If $|\alpha| = 0$ and $|\beta|\ge 1$, then $\alpha = r(\beta)$. Write $\beta = \beta_m\cdots \beta_2\beta_1$, where $\beta_i\in E^1$ for all $1\le i\le m$. Let $q := f_m\cdots f_2f_1$, where $f_i = \sigma^i_1(\beta_i)$ for all $1\le i\le m$. Since $\beta$ is a path in $E$ and $\sigma^m$ is an automorphism of $E$, $\sigma^m_1(\beta_m)\cdots \sigma^m(\beta_1)$ is also a path in $E$. Therefore, by Lemma \ref{lm:twsist-path}, $q$ is a path in $E_{\sigma}$. Then, using formula $(1)$, we obtain that $$\Theta(q^*) = B_{f^*_1}\star B_{f^*_2} \cdots \star B_{f^*_m}= \sigma^{-1}_1(f_1)^*\sigma^{-2}_1(f_2)^*\cdots\sigma^{-m}_1(f_m)^*= f_1^*f_2^*\cdots f_m^*=\beta^*,$$ and hence $\beta^*\in \text{Im}(\Theta)$, as desired.

If $|\alpha|\ge 1$ and $|\beta|\ge 1$, then we express 
\begin{center}
$\alpha = \alpha_1\cdots \alpha_l$\quad and \quad $\beta^* = \beta_t^*\cdots \beta_1^*$,    
\end{center}
where $\alpha_i, \beta_j\in E^1$ for all $i, j$. Let 
\begin{center}
$x = e_1\cdots e_l$\quad and \quad $y = f_m\cdots f_1$,    
\end{center}
where $e_i = \sigma^{-(i-1)}_1(\alpha_i)$ and $f_j = \sigma^{j-l}_1(\beta_j)$ for all $1\le i\le l$ and $1\le j\le m$.
Then, using formula $(1)$, we have 
\begin{align*}
\Theta(xy^*) = A_{e_1}\star\cdots\star A_{e_l}\star B_{f^*_1}\star \cdots \star B_{f^*_t}
                       & = e_1\sigma_1(e_2)\cdots\sigma_1^{l-1}(e_l)\sigma^{l-1}_1(f_1)^*\cdots \sigma^{l-t}_1(f_t)^*\\
                       & = \alpha_1\cdots \alpha_l\beta_1^*\cdots \beta_t^*=\alpha\beta^*, \end{align*}
so $\alpha\beta^*\in \text{Im}(\Theta)$, and hence $\Theta$ is surjective. Therefore, $\Theta$ is a graded $K$-algebra isomorphism.

If $E^0$ is finite, then  $1_{L_K(E_{\sigma})} = \sum_{v\in E^0} = 1_{(L_K(E))^{\sigma}}$, and hence $$\Theta(1_{L_K(E_{\sigma})}) = \sum_{v\in E^0}\Theta(v) = \sum_{v\in E^0}v = 1_{(L_K(E))^{\sigma}},$$
thus completing the proof.                    
\end{proof}

For clarification, we illustrate Theorem \ref{thm:twist-is-LPA}  by presenting the following example.

\begin{example} \label{ex:C2}
Consider the graph $E=C_2$: \[\begin{tikzcd}
	{\bullet v_1} && {\bullet v_2}
	\arrow["{e_1}", from=1-1, to=1-3]
	\arrow["{e_2}", shift left=3, from=1-3, to=1-1]
\end{tikzcd}\]
and let $\sigma$ be the automorphism of $E$ introduced in Example \ref{ex:twist-quiver}. We then have  $E_\sigma=R_1\sqcup R_1$:
\[\begin{tikzcd}
	{\bullet v_1} && {\bullet v_2}
	\arrow["{e_1}", from=1-1, to=1-1, loop, in=55, out=125, distance=10mm]
	\arrow["{e_2}", from=1-3, to=1-3, loop, in=55, out=125, distance=10mm]
\end{tikzcd}\]
Theorem~\ref{thm:twist-is-LPA} shows that
\[
L_\K(C_2)^{\sigma}\;\cong L_\K((C_2)_{\sigma}) \cong\;\K[x,x^{-1}]\times\K[x,x^{-1}] .
\]
Thus, the twisted Leavitt path algebra is drastically different from the original Leavitt path algebra. Indeed, $$L_\K(C_2)\cong M_2(\K[x, x^{-1}]),$$ which is noncommutative, whereas $L_\K(C_2)^{\sigma}$ is commutative. 
\end{example}

Theorem \ref{thm:twist-is-LPA} requires $\sigma$ to come from an automorphism of $E$; in particular, $\sigma$ permutes the vertices. We now show that this hypothesis cannot be dropped: for a graded automorphism that
does not permute the vertices, the twisted algebra need not be a Leavitt path algebra at all. We first characterize which Leavitt path algebras are commutative.

\begin{lemma}\label{lem:commutative-LPA}
Let $K$ be a field and $E$ be a graph such that $\LE$ is commutative. Then every edge of $E$ is a
loop, every vertex emits at most one edge, and
\[
\LE\;\cong\;\bigoplus_{v\in E^0}R_v,\qquad\text{where }R_v\cong K\text{ or }R_v\cong K[x,x^{-1}].
\]
Consequently, for every idempotent $\epsilon\in\LE$, the corner $\epsilon\,\LE\,\epsilon$ has
only finitely many idempotents.
\end{lemma}

\begin{proof}
Let $e\in E^1$. If $s(e)\neq r(e)$, then
$e=s(e)e=e\,s(e)=e\,r(e)\,s(e)=0$, a contradiction, since every edge is nonzero in $\LE$.
Hence every edge is a loop. Suppose that $e\neq f$ are loops at the same vertex $v$. Then
$e^*e=v$, and commutativity gives $ee^*=e^*e=v$. Therefore
$f=vf=ee^*f=e(e^*f)=0$, a contradiction. Thus every vertex $v$ emits at most one edge, and
$E$ is a disjoint union of isolated vertices and single loops. It follows that
$\LE=\bigoplus_{v\in E^0}R_v$, where $R_v=Kv\cong K$ if $v$ emits no edge, and
$R_v\cong K[x,x^{-1}]$ (via $e\mapsto x$, $e^*\mapsto x^{-1}$) if $v$ carries a loop $e$.

Since each $R_v$ is a domain, its only idempotents are $0$ and $1_{R_v}=v$. Hence every
idempotent $\epsilon$ of $\bigoplus_v R_v$ has the form $\epsilon=\sum_{v\in S}v$ for a finite
set $S\subseteq E^0$, and $\epsilon\,\LE\,\epsilon\cong\bigoplus_{v\in S}R_v$ has exactly
$2^{|S|}$ idempotents.
\end{proof}

\begin{example}\label{ex:twist-not-LPA}
Let $E$ be the graph consisting of a left-infinite tail attached to the root of the infinite
binary tree. Precisely, the vertices of $E$ are
\[
u_n\ (n\le -1)\qquad\text{and}\qquad v_w\ \ (w\in\{0,1\}^{*}),
\]
where $\{0,1\}^*$ denotes the set of finite words in $0$ and $1$, including the empty word
$\emptyset$. The edges of $E$ are
\[
g_n\colon u_n\to u_{n+1}\ (n\le -2),\qquad g_{-1}\colon u_{-1}\to v_\emptyset,\qquad
e_{w,a}\colon v_w\to v_{wa}\ \ (w\in\{0,1\}^*,\ a\in\{0,1\}).
\]
Thus $E$ is row-finite, it has neither sinks nor sources, and every vertex is regular.

\smallskip
\noindent Let $C$ be the $K$-algebra of locally constant functions $\{0,1\}^{\N}\to K$, and for
$w\in\{0,1\}^*$, let $\chi_w\in C$ be the characteristic function of the set of all
sequences beginning with $w$. Then $\chi_w=\chi_{w0}+\chi_{w1}$, and $C$ is spanned by the
$\chi_w$. Let
\[
M=C\otimes_K M_\Z(K),
\]
where $M_\Z(K)$ is the algebra of $\Z\times\Z$ matrices over $K$ with only finitely many nonzero
entries, with matrix units $E_{ij}$. We grade $M$ by $\deg(f\otimes E_{ij})=j-i$, and write
$(f\otimes E_{ij})^*=f\otimes E_{ji}$. For $w\in\{0,1\}^*$, we write $|w|$ for its length.
Consider the elements
\[
\pi(u_n)=1\otimes E_{nn},\qquad \pi(v_w)=\chi_w\otimes E_{|w|,|w|},\qquad
\pi(g_n)=1\otimes E_{n,n+1},\qquad \pi(e_{w,a})=\chi_{wa}\otimes E_{|w|,|w|+1},
\]
and $\pi(x^*)=\pi(x)^*$ for $x\in E^1$. These elements satisfy relations (1)--(4) of
Definition \ref{LPAs}. The only relation that is not immediate is~(4) at the vertex $v_w$, which reads
\[
\sum_{a\in\{0,1\}}\pi(e_{w,a})\pi(e_{w,a})^*=\sum_{a\in\{0,1\}}\chi_{wa}\otimes E_{|w|,|w|}
=\chi_w\otimes E_{|w|,|w|}=\pi(v_w),
\]
together with $\pi(e_{w,a})^*\pi(e_{w,a})=\chi_{wa}\otimes E_{|w|+1,|w|+1}=\pi(v_{wa})$. By the
Universal Property of $\LE$, $\pi$ extends to a $\Z$-graded $K$-algebra homomorphism
$\pi\colon\LE\to M$. Since $\pi(v)\neq0$ for every $v\in E^0$, $\pi$ is injective by the Graded
Uniqueness Theorem (see, e.g., \cite[Theorem 2.2.15]{AAS}).

To see that $\pi$ is surjective, let $w\in\{0,1\}^*$ and $m\le d:=|w|$. Let $p$ be the unique
path in $E$ ending at $v_w$ whose source is $u_m$ if $m<0$, and the vertex $v_{w'}$ with $w'$ the
prefix of $w$ of length $m$ if $m\ge0$. Multiplying the images of its edges gives
$\pi(p)=\chi_w\otimes E_{m,d}$. Hence, for all $i,j\le|w|$,
\[
\chi_w\otimes E_{ij}=\pi(p)\,\pi(q)^*\in\operatorname{Im}(\pi),
\]
where $p$ and $q$ are the paths to $v_w$ starting at depth $i$ and $j$, respectively. Since
$\chi_w=\sum_{w''}\chi_{ww''}$, the sum taken over all words $w''$ of a fixed length, every
element $\chi_w\otimes E_{ij}$ is a sum of elements $\chi_{w'}\otimes E_{ij}$ with
$|w'|\ge\max(i,j)$. Thus $\operatorname{Im}(\pi)=M$, and $\pi$ is an isomorphism of $\Z$-graded
$K$-algebras.

\smallskip

\noindent The map $S\colon M\to M$, $S(f\otimes E_{ij})=f\otimes E_{i+1,j+1}$, is a graded algebra
automorphism of $M$ of infinite order. Let $\sigma=\pi^{-1}\circ S\circ\pi$. Then $\sigma$ is a graded algebra
automorphism of $\LE$. Explicitly, $\sigma(x^*)=\sigma(x)^*$ for $x\in E^1$, and
\[
\sigma(u_n)=u_{n+1}\ (n\le-2),\qquad \sigma(u_{-1})=v_\emptyset,\qquad
\sigma(v_w)=v_{w0}+v_{w1},
\]
\[
\sigma(g_n)=g_{n+1}\ (n\le-2),\qquad \sigma(g_{-1})=e_{\emptyset,0}+e_{\emptyset,1},\qquad
\sigma(e_{w,a})=e_{wa,0}+e_{wa,1}.
\]
For instance, $\pi(\sigma(e_{w,a}))=\chi_{wa}\otimes E_{|w|+1,|w|+2}
=\sum_{b}\chi_{wab}\otimes E_{|wa|,|wa|+1}=\pi(e_{wa,0}+e_{wa,1})$. Since $\sigma(v_w)$ is not a
vertex, $\sigma$ is not induced by an automorphism of $E$.

\smallskip

\noindent Since $\pi$ is a graded isomorphism with $\pi\circ\sigma=S\circ\pi$, it is also an isomorphism
$\LE^\sigma\cong M^S$. For $f,g\in C$, $i,k\in\Z$ and $n,m\in\Z$, we have
\begin{align*}
(f\otimes E_{i,i+n})\star(g\otimes E_{k,k+m})
&=(f\otimes E_{i,i+n})\,S^n(g\otimes E_{k,k+m})\\
&=(f\otimes E_{i,i+n})(g\otimes E_{k+n,k+m+n})
=\delta_{i,k}\,fg\otimes E_{i,i+n+m}.
\end{align*}
Hence $M^S=\bigoplus_{i\in\Z}M^{(i)}$ as algebras, where
$M^{(i)}=\operatorname{span}_K\{f\otimes E_{i,j}\mid f\in C,\ j\in\Z\}$, and
$f\otimes E_{i,i+n}\mapsto f\,t^n$ is an isomorphism $M^{(i)}\cong C[t,t^{-1}]$. Therefore
\[
\LE^\sigma\;\cong\;\bigoplus_{i\in\Z}C[t,t^{-1}],
\]
and in particular $\LE^\sigma$ is commutative, whereas $\LE\cong C\otimes M_\Z(K)$ is not.

\smallskip

\noindent The vertex $v_\emptyset$ is an idempotent of $\LE^\sigma$, since $v_\emptyset\star
v_\emptyset=v_\emptyset^2=v_\emptyset$. Under the above isomorphism it corresponds to the identity of
$M^{(0)}$, so
\[
v_\emptyset\star\LE^\sigma\star v_\emptyset\;\cong\;C[t,t^{-1}].
\]
This corner contains the infinitely many distinct idempotents $\chi_w$, $w\in\{0,1\}^*$. These
are the elements $p_wp_w^*$, where $p_w$ is the path from $v_\emptyset$ to $v_w$; they are
idempotents in $\LE^\sigma$ because $\star$ agrees with the original product on
$\LE_0$. By Lemma~\ref{lem:commutative-LPA}, no commutative Leavitt path algebra has an
idempotent whose corner contains infinitely many idempotents. Consequently,
$\LE^\sigma$ is not isomorphic to $L_K(F)$, even as an ungraded $K$-algebra, for any graph
$F$.
\end{example}

\begin{remark}
\begin{enumerate}
\item If $C$ is replaced by $K$, that is, if the binary tree is replaced by the right-infinite
tail $v_\emptyset\to v_0\to v_{00}\to\cdots$, then $E$ becomes the graph of Example \ref{ex:inf-graph} and
$S$ becomes the automorphism induced by the shift $\sigma_0(v_n)=v_{n+1}$. In that case
Theorem \ref{thm:twist-is-LPA} applies and $\LE^\sigma\cong\bigoplus_{\Z}K[x,x^{-1}]$ is again a Leavitt path
algebra. In Example~\ref{ex:twist-not-LPA}, the binary branching is exactly what prevents $S$
from permuting the vertices.

\item The graph in Example~\ref{ex:twist-not-LPA} is infinite, so $\LE$ is not unital. We do
not know whether the Zhang twist of a unital Leavitt path algebra by a graded automorphism can
fail to be a Leavitt path algebra.
\end{enumerate}
\end{remark}

\section{A noncommutative Proj for $\mathbb{Z}$-graded algebras}\label{noncom-Proj}

\noindent The noncommutative projective scheme of Artin and Zhang \cite{AZ} is
defined for connected $\mathbb{N}$-graded algebras, whereas the Leavitt path
algebra $L_K(E)$ of a graph $E$ over a field $K$ is $\mathbb{Z}$-graded, with its degree-zero component typically much larger than $K$. 
We therefore begin by introducing a version of
$\proj$ for arbitrary $\mathbb{Z}$-graded algebras. This construction recovers the Artin--Zhang construction in the connected $\mathbb{N}$-graded case. Throughout, let $A=\bigoplus_{n\in\mathbb{Z}}A_n$
be a unital $\mathbb{Z}$-graded $K$-algebra, and assume that all modules are right modules.

Consider the following two-sided ideal
\[
  I_A =\; A\Big(\bigoplus_{n\neq 0}A_n\Big)A.
\]
Since $I_A$ is generated by homogeneous elements, it is a graded ideal.

Recall that a full subcategory of a Grothendieck category is
\emph{localizing} if it is closed under subobjects, quotients,
extensions, and arbitrary direct sums. The intersection of any collection of localizing subcategories is again localizing.

\begin{definition}\label{def:proj}
Let $\Tors_{I}(A)$ be the smallest localizing subcategory of $\Gr\text{-}A$
that contains every graded module $M$ with $MI_A=0$. We put
\[\mathsf{QGr}_I\text{-}A =\; \mathsf{Gr}\text{-}A\big/\Tors_I(A),
  \qquad
  \pi\colon \mathsf{Gr}\text{-}A\longrightarrow \mathsf{QGr}_I\text{-}A
\]
for the Gabriel quotient \cite{Gabriel} and the quotient functor, respectively. The
\emph{noncommutative projective scheme} of $A$ is the triple
\[
  \proj A =\; \big(\mathsf{QGr}_I\text{-}A,\ \pi(A_A),\ s\big),
\]
where $s$ is the autoequivalence induced by the degree shift
$M\mapsto M(1)$. As in \cite{AZ}, two such triples
$(\mathcal C,\mathcal O,s)$ and $(\mathcal C',\mathcal O',s')$ are
\emph{equivalent} if there exists an equivalence
$G\colon\mathcal C\to\mathcal C'$ with $G(\mathcal O)\cong\mathcal O'$
and a natural isomorphism $G\circ s\cong s'\circ G$.
\end{definition}

The shift is well defined on the quotient. Indeed, $M\mapsto M(1)$ is
an autoequivalence of $\mathsf{Gr}\text{-}A$, and $M(1)I_A=0$ if and only if
$MI_A=0$. So the shift maps the generating class of $\Tors_I(A)$ onto
itself, and hence preserves $\Tors_I(A)$. Since $\mathsf{Gr}\text{-}A$ is a
Grothendieck category, so is $\mathsf{QGr}_I\text{-}A$, and $\pi$ is exact
with a right adjoint.

\begin{proposition}\label{prop:AZ-recovered}
Let $A$ be a connected $\mathbb{N}$-graded, right noetherian $K$-algebra. We may regard $A$ as a $\mathbb{Z}$-graded algebra with $A_n=0$ for $n<0$. Then $\Tors_I(A)$ defined as above coincides with the torsion subcategory of Artin--Zhang. Consequently, $\proj A$ is the noncommutative projective scheme in the sense of \cite{AZ}.
\end{proposition}

\begin{proof}
Here $I_A=A_{\ge 1}$, since $A_{\ge 1}$ is already a
two-sided ideal. Let $\mathcal T$ be the Artin--Zhang torsion class.
It consists of the graded modules in which every element $m$ satisfies
$mA_{\ge n}=0$ for some $n$, and it is a localizing subcategory \cite[\S2]{AZ}. Since $\mathcal T$
contains every module annihilated by $A_{\ge 1}$, we have
$\Tors_I(A)\subseteq\mathcal T$.

For the reverse inclusion, let $M\in\mathcal T$, and let $m\in M$ be homogeneous, and choose $n$ such that $mA_{\ge n}=0$. Since
$A_{\ge 1}^{n}\subseteq A_{\ge n}$, the graded submodule $N=mA$
satisfies $NA_{\ge 1}^{n}=m\,A_{\ge 1}^{n}=0$. Thus
\[
  N\supseteq NA_{\ge 1}\supseteq\cdots\supseteq NA_{\ge 1}^{n}=0
\]
is a finite filtration by graded submodules whose successive factors are annihilated by $A_{\ge 1}$. Hence $N\in\Tors_I(A)$, since $\Tors_I(A)$ is closed under
extensions. Finally, $M$ is a quotient of the direct sum of its submodules $mA$, where $m$ ranges over the homogeneous elements of $M$, and so $M\in\Tors_I(A)$.
\end{proof}

\begin{proposition}\label{prop:strong}
Let $A$ be a strongly $\Z$-graded $K$-algebra. Then $I_A=A$ and
$\Tors_I(A)=0$. Consequently, the functor $M\mapsto M_0$ induces an
equivalence
\[
  \proj A\;\simeq\;\big(\mathsf{Mod}\text{-}A_0,\ A_0,\ -\otimes_{A_0}A_1\big).
\]
\end{proposition}

\begin{proof}
First, $1\in A_1A_{-1}\subseteq I_A$, so $I_A=A$. If $MI_A=0$, then
$M=M\cdot1=0$, so $\Tors_I(A)=0$ and $\mathsf{QGr}_I\text{-}A=\mathsf{Gr}\text{-}A$.

By Dade's theorem \cite{D}, $M\mapsto M_0$ is an equivalence
$\mathsf{Gr}\text{-}A\to \mathsf{Mod}\text{-}A_0$ with quasi-inverse $N\mapsto N\otimes_{A_0}A$. It
sends $A_A$ to $A_0$. Moreover, $M(1)_0=M_1$, and for strongly graded
$A$ the multiplication map $M_0\otimes_{A_0}A_1\to M_1$ is an
isomorphism.
\end{proof}

By \cite[Theorem~3.15]{H}, $\LE$ is strongly graded for every
finite graph $E$ without sinks. Hence, for such $E$,

\begin{equation}
  \proj \LE\;\simeq\;\big(\mathsf{Mod}\text{-}L_K(E)_0,\ L_K(E)_0,\ -\otimes_{L_K(E)_0}\LE_1\big).
\end{equation}

We now show that this notion of $\proj$ for $\mathbb Z$-graded algebras is well adapted to Zhang twists by graded automorphisms. For a graded automorphism $\sigma$ of $A$, write $\bar\sigma=\sigma|_{A_0}$. For a right $A_0$-module $N$, let $\bar\sigma^*N$ denote $N$ with the action $n\cdot b=n\,\bar\sigma(b)$.
Recall that Zhang's equivalence $F\colon \mathsf{Gr}\text{-}A\to \mathsf{Gr}\text{-}A^{\sigma}$
\cite{Zhang96} sends $M$ to the same graded vector space with action
\[
  m\star a=m\,\sigma^{n}(a)\qquad (m\in M_n).
\]

\begin{theorem}\label{thm:proj-twist}
Let $\sigma$ be a graded automorphism of $A$. Then the following statements hold:
\begin{enumerate}
\item $I_{A^\sigma}=I_A$ as graded subspaces of $A$.
\item $F$ restricts to an equivalence
  $\Tors_I(A)\simeq\Tors_I(A^{\sigma})$. It therefore induces an
  equivalence $\mathsf{QGr}_I\text{-}A\simeq \mathsf{Mod}_I\text{-}A^{\sigma}$ sending
  $\pi(A_A)$ to $\pi(A^\sigma_{A^\sigma})$.
\item If $A$ is strongly graded, then so is $A^{\sigma}$, and
  $A^\sigma_0=A_0$ as algebras. Under the identifications of
  Proposition~\ref{prop:strong}, the equivalence in (2) is the identity
  functor of $\mathsf{Mod}\text{-}A_0$, and
  \[
\operatorname{proj}A^\sigma\;\simeq\;\bigl(\mathsf{Mod}\text{-}A_0,\ A_0,\
\bar\sigma^*\circ(-\otimes_{A_0}A_1)\bigr),
\]

  Thus the passage from $A$ to $A^\sigma$ leaves the pair
  $(\mathsf{Mod}\text{-}A_0,A_0)$ unchanged and composes the shift with
  $\bar\sigma^*$.
\end{enumerate}
\end{theorem}

\begin{proof}
(1) Let $S=\bigoplus_{n\ne0}A_n$. For homogeneous $a\in A_p$,
$s\in A_q$ with $q\ne0$, and $b\in A_r$, we have
\[
  a\star s\star b=a\,\sigma^{p}(s)\,\sigma^{p+q}(b).
\]
Since each $\sigma^k$ restricts to a bijection of every $A_j$, the span
of these elements is the span of the products $A_pA_qA_r$ with
$q\ne0$, which is $I_A$.

(2) By (1), $I=I_A=I_{A^\sigma}$, and $I$ is $\sigma$-stable because
$S$ is. For $M\in \mathsf{Gr}\text{-}A$, we get
\[
  F(M)\star I=\operatorname{span}\{m\,\sigma^{n}(x): m\in M_n,\ x\in I\}=MI.
\]
So $MI=0$ if and only if $F(M)\star I=0$. An equivalence of
Grothendieck categories carries localizing subcategories to localizing
subcategories, so $F$ maps the smallest localizing subcategory
containing the modules annihilated by $I$ onto the corresponding one
for $A^\sigma$. The universal property of the Gabriel quotient
\cite{Gabriel} then gives the induced equivalence. Finally, $F(A_A)$ is $A^\sigma_{A^\sigma}$.

(3) First, $A^\sigma_n\star A^\sigma_m=A_n\sigma^n(A_m)=A_nA_m$, so
$A^\sigma$ is strongly graded. For $a,b\in A_0$ we have
$a\star b=ab$, so $A^\sigma_0=A_0$. Also
$F(M)_0=M_0$ with $m\star a=ma$ for $a\in A_0$, so $F$ becomes the
identity under $M\mapsto M_0$.

The $A_0$-bimodule $A^\sigma_1$ is $A_1$ with left action $b\star x=bx$ and right action $x\star b=x\,\bar\sigma(b)$. Hence,
naturally in $N$, $N\otimes_{A_0}A_1^\sigma\cong\bar\sigma^*(N\otimes_{A_0}A_1)$. Now applying Proposition~\ref{prop:strong} to $A^\sigma$ gives us the result.
\end{proof}

Combining Theorem~\ref{thm:proj-twist} with Theorem~\ref{thm:twist-is-LPA}, we obtain the
following.

\begin{corollary}\label{cor:proj-LPA-twist}
Let $E$ be a finite graph without sinks, and let $\sigma$ be an
automorphism of $E$. Then $(\LE^{\sigma})_0 \cong \LE_0$, and
\[
\proj \LE_\sigma \;\simeq\;\proj \LE^\sigma\;\simeq\;
\bigl(\mathsf{Mod}\text{-}\LE_0,\ \LE_0,\ \bar\sigma^*\circ(-\otimes_{\LE_0}\LE_1)\bigr),
\]
where $\bar\sigma$ is the restriction of $\sigma$ to $\LE_0$. In particular, $\proj \LE$ and $\proj \LE^\sigma$ have the same
underlying pair $(\mathcal C,\mathcal O)$, and they differ at most in
the shift functor.
\end{corollary}

The shift can be made completely explicit when $A$ has a unit of degree one.

\begin{proposition}\label{prop:degree-one-unit}
Suppose that $A$ contains a unit $t$ of degree~$1$, and put
$\alpha(a)=tat^{-1}$ for $a\in A_0$. Then the following statements hold:
\begin{enumerate}
\item $A=A_0[t,t^{-1};\alpha]$ is a skew Laurent polynomial ring, and
  $A$ is strongly graded.
\item $-\otimes_{A_0}A_1\cong\alpha^{*}$, and therefore
  $\proj A\simeq(\mathsf{Mod}\text{-}A_0,A_0,\alpha^{*})$.
\item For every graded automorphism $\sigma$ of $A$, the element $t$ is
  a unit of $A^\sigma$, and
 \[
A^\sigma\cong A_0[u,u^{-1};\alpha\circ\bar\sigma],\qquad
\operatorname{proj}A^\sigma\simeq\bigl(\mathsf{Mod}\text{-}A_0,\ A_0,\ (\alpha\circ\bar\sigma)^*\bigr).
\]
\end{enumerate}
\end{proposition}

\begin{proof}
(1) For $x\in A_n$ we have $x=(xt^{-n})t^n$ with $xt^{-n}\in A_0$, and
$at^n=0$ forces $a=0$. Hence $A_n=A_0t^n$ is free of rank one as a left
$A_0$-module. Moreover $t^na=\alpha^n(a)t^n$, and
$1=t^nt^{-n}\in A_nA_{-n}$.

(2) By (1), $N\otimes_{A_0}A_1\to N$, $n\otimes at\mapsto na$, is a
bijection. It transports the right action to
\[
  (n\otimes t)b=n\otimes\alpha(b)t,
\]
that is, to $\alpha^{*}N$. Now apply Proposition~\ref{prop:strong}.

(3) The element $t'=\sigma^{-1}(t^{-1})$ has degree $-1$, and
\[
  t\star t'=t\,\sigma(t')=1,
  \qquad
  t'\star t=t'\,\sigma^{-1}(t)=\sigma^{-1}(t^{-1}t)=1.
\]
For $a\in A_0$,
\[
  t\star a\star t'=t\,\sigma(a)\,\sigma(t')=t\,\sigma(a)\,t^{-1}=\alpha(\sigma(a)).
\]
Apply (1) and (2) to $A^\sigma$. This agrees with
Theorem~\ref{thm:proj-twist}(3), since
$\bar\sigma^*\circ\alpha^*=(\alpha\circ\bar\sigma)^*$.
\end{proof}

We now show how the twist changes the shift functor. The following elementary observation is all that is needed to distinguish the triples associated to a Leavitt path algebra and its twist.

\begin{lemma}\label{lem:rigid}
Let $\mathcal C$ and $\mathcal C'$ be categories, let $s$ and $s'$ be
autoequivalences of $\mathcal C$ and $\mathcal C'$ respectively, and
suppose that there is an equivalence $G\colon\mathcal C\to\mathcal C'$
with $G\circ s\cong s'\circ G$. If $s'\cong \text{id}_{\mathcal C'}$, then
$s\cong \text{id}_{\mathcal C}$.
\end{lemma}

\begin{proof}
Let $G^{-1}$ be a quasi-inverse of $G$. Since $G\circ s\cong s'\circ
G\cong G$, we obtain that
\[
  s\;\cong\;G^{-1}\circ G\circ s\;\cong\;G^{-1}\circ G\;\cong\;\text{id}_{\mathcal C}. 
\]\end{proof}

\begin{example}\label{ex:Cm}
Let $E=C_m$ be the graph of Example \ref{C_m} with vertices
$v_1,\dots,v_m$ and edges $e_i\colon v_i\to v_{i+1}$, where the
indices are taken modulo $m$, and let $\sigma$ be the rotation
$v_i\mapsto v_{i+1}$, $e_i\mapsto e_{i+1}$.

Paths in $C_m$ are determined by their source and length, so
$\LE_0=\bigoplus_i Kv_i\cong K^m$. The element $t=\sum_i e_i$ is a
unit of degree $1$, with $t^{-1}=\sum_i e_i^{*}$. A direct computation
gives
\[
  t\,v_i\,t^{-1}=e_{i-1}e_{i-1}^{*}=v_{i-1} .
\]
Hence $\alpha=\rho$, the cyclic permutation $v_i\mapsto v_{i-1}$ of
$K^m$, and $\bar\sigma=\rho^{-1}$.

By Proposition \ref{prop:degree-one-unit}, we have
\[
  \LE\cong K^m[t,t^{-1};\rho],
  \qquad
  \LE^{\sigma}\cong K^m[u,u^{-1}]=K[u,u^{-1}]^m .
\]
Moreover,
\begin{equation}\label{eq:two-triples}
  \proj\LE\simeq\big(\mathsf{Mod}\text{-}K^m,\ K^m,\ \rho^{*}\big),
  \qquad
  \proj\LE^{\sigma}\simeq\big(\mathsf{Mod}\text{-}K^m,\ K^m,\ \text{id}\big).
\end{equation}

For $1\le i\le m$ let $S_i$ denote the simple right $K^m$-module on
which $v_i$ acts as the identity and $v_j$ acts as $0$ for $j\neq i$.
These are pairwise non-isomorphic, and they exhaust the simple
$K^m$-modules. For an automorphism $\gamma$ of $K^m$ we write
$\gamma^{*}N$ for $N$ equipped with the action $n\cdot b=n\,\gamma(b)$.

The element $v_j$ acts on $\rho^{*}S_i$ in the same way as
$\rho(v_j)=v_{j-1}$ acts on $S_i$, and the latter is the identity
precisely when $j=i+1$. Hence
\[
  \rho^{*}S_i\;\cong\;S_{i+1}\qquad(1\le i\le m),
\]
the indices again taken modulo $m$. Since $m\ge2$, we have
$S_{i+1}\not\cong S_i$, and therefore
$\rho^{*}\not\cong \text{id}_{\mathsf{Mod}\text{-}K^m}$.

By \eqref{eq:two-triples} the shift of $\proj\LE^{\sigma}$ is the
identity functor. An equivalence of the two triples is in particular
an equivalence $G$ of the underlying categories with
$G\circ\rho^{*}\cong\text{id}\circ\,G$, so Lemma~\ref{lem:rigid} would force
$\rho^{*}\cong \text{id}$, a contradiction. This shows that $\proj \LE$ and $\proj \LE^\sigma$ have the same underlying pair $(\mathcal C,\mathcal O)$, but they differ in the shift functor.
\end{example}

\begin{corollary}\label{cor:not-gr-morita}
Let $m\ge2$. Then $L_K(C_m)$ and $L_K(C_m)^{\sigma}$ are
$\Gr$-equivalent but not graded Morita equivalent.
\end{corollary}
\begin{proof}
The $\Gr$-equivalence follows from Zhang's theorem \cite{Zhang96}, or alternatively from Theorem \ref{thm:proj-twist}(2) together with Proposition \ref{prop:strong}, since both algebras are strongly graded.

Recall that two graded algebras are graded Morita equivalent if and only if their graded module categories are equivalent via an equivalence that commutes with the degree-shift functor.

Suppose, for contradiction, that there exists such an equivalence $$H\colon \mathsf{Gr}\text{-}L_K(C_m)\to
\mathsf{Gr}\text{-}L_K(C_m)^{\sigma}.$$  Transporting $H$
through the equivalences in Proposition \ref{prop:strong} would yield an equivalence
$$G: \mathsf{Mod}\text{-}K^m \longrightarrow \mathsf{Mod}\text{-}K^m$$ such that $G\circ\rho^{*}\cong \text{id}\circ\,G$, which is
impossible. Hence $L_K(C_m)$ and $L_K(C_m)^{\sigma}$ are not graded Morita equivalent.
\end{proof}

\subsection*{Connection with the work of Smith in \texorpdfstring{\cite{S}}{Smith}}

\noindent For any finite graph $E= (E^0, E^1)$ we denote by $A_E$ the adjacency matrix of $E$. More precisely, if $E^0 = \{1, 2, \ldots, n\}$, then $A_E = ((A_E)_{i,j})$ is the $n\times n$ matrix whose $(i, j)$-entry $(A_E)_{i,j}$ is the number of edges $e$ in $E^1$ such that $s(e) = i$ and $r(e)= j$. 
Let $P=[P_{i,j}]\in M_{n}(\mathbb{N})$
be the permutation matrix associated to $\sigma_0$, defined by $$P_{i,j}=\delta_{j,\sigma_0(i)}$$ for all $1\le i, j\le n$. 

The following fact describes the relationship between the adjacency matrices of $E$ and $E_{\sigma}$.

\begin{proposition}\label{prop:adjacency}
Let $E$ be a finite graph and let $\sigma$ be an automorphism of
$E$, and let $m$ be the order of $\sigma_0$. Then
\[
A_{E_\sigma}=A_E\,P^{-1},\qquad PA_EP^{-1}=A_E ,
\]
so $A_E$ and $P$ commute and consequently
\[
A_{E_\sigma}^{\,k}=A_E^{\,k}\,P^{-k}\quad (k\ge 0),
\qquad\text{and in particular}\qquad
A_{E_\sigma}^{\,m}=A_E^{\,m}.
\]
\end{proposition}

\begin{proof}
Since $r_\sigma(e)= j$ if and only if $r(e)=\sigma_0(j)$, and so we obtain that
$$(A_{E_\sigma})_{i,j}=(A_E)_{i,\sigma_0(j)}=(A_EP^{-1})_{i,j}.$$
Moreover, since $\sigma_1$ restricts to a bijection from the set of edges $i\to j$ onto the set of edges
$\sigma_0(i)\to\sigma_0(j)$, we have $$(A_E)_{\sigma_0(i),\sigma_0(j)}=(A_E)_{i,j}$$ for all
$i,j$. This is precisely the matrix identity $$PA_E=A_EP.$$ Consequently, 
$$A_{E_\sigma}^{\,k}=(A_EP^{-1})^k=A_E^{\,k}P^{-k}$$ for all $k\ge 0$. Since $P^m=I$, the desired final equality follows, thus completing the proof.
\end{proof}

For a graph $E = (E^0, E^1)$, let $KE$ denote its path algebra over a field $K$, equipped with the $\mathbb{N}$-grading by path length. The path algebra $KE$ is also considered as a $\mathbb{Z}$-graded $K$-algebra  by setting  $(KE)_n = 0$ for all $n < 0$. If $E$ is finite, then Smith proved in \cite[Theorems 1.3 and 1.8]{S} that
\[
\mathsf{QGr}\text{-}KE\;\equiv\;\mathsf{Mod}\text{-}S(E)\;\equiv\;\mathsf{Gr}\text{-}L_K(E^{\circ})
\;\equiv\;\mathsf{Mod}\text{-}L_K(E^{\circ})_0\;\equiv\;\mathsf{QGr}\text{-} KE^{(n)}
\tag{4}
\] for all $n\ge 1$, where $S(E)=\varinjlim\operatorname{End}_{KE^0}\bigl((KE^1)^{\otimes n}\bigr)$ is a direct limit
of finite dimensional semisimple algebras, $E^{\circ}$ is the graph without sinks or sources obtained from $E$ by repeatedly deleting all sinks and all sources, and $E^{(n)}$ is the graph with vertex set $E^0$ whose adjacency matrix is the $n$-th power of that of $E$. Thus, provided $E= E^{\circ}$, $\mathsf{Gr}\text{-}L_K(E)$ is itself a noncommutative $\operatorname{Proj}$, namely  the noncommutative 
$\operatorname{Proj}$ of the path algebra $KE$.

It is worth observing that the twist in Theorem \ref{thm:twist-is-LPA} is compatible with every term in $(4)$. First, it commutes with the passage from $E$ to $E^{\circ}$.

\begin{lemma}\label{lem:circ-commutes}
Let $E$ be a finite graph and $\sigma$ an automorphism of $E$. Then $E$ and $E_\sigma$ have
the same sinks and the same sources, and
\[
(E_\sigma)^{\circ}=\bigl(E^{\circ}\bigr)_{\sigma},
\]
where $\sigma$ also denotes the induced automorphism of $E^{\circ}$.
\end{lemma}

\begin{proof}
The graphs $E$ and $E_\sigma$ have the same source map and, hence, have the same sinks. A vertex $w$ receives an edge in $E_{\sigma}$ if and only if $\sigma_0(w)$ receives an edge in
$E$. Since $\sigma$ is an automorphism of $E$, the set of sources of $E$ is $\sigma_0$-invariant. Therefore, $w$ is a source of $E_{\sigma}$ if and only if it is a source of $E$.
Hence, the set $S$ of vertices deleted at the first stage is the same for $E$ and for
$E_{\sigma}$, and $S$ is $\sigma_0$-invariant. An edge $e$ survives in $E$ if and only if
\begin{center}
$s(e)\notin S$ and $r(e)\notin S$,    
\end{center}
whereas it survives in $E_{\sigma}$ if and only if 
\begin{center}
$s(e)\notin S$ and $r_{\sigma}(e)=\sigma_0^{-1}r(e)\notin S$.  
\end{center}
These conditions are equivalent because $S$ is
$\sigma_0$-invariant. Thus, the two full subgraphs obtained after one stage are again related by the twist along the restricted automorphism, and the claim follows by induction on the number of stages.
\end{proof}

For a field $K$ and a graph $E$, the path algebra $KE$ has  $\text{Path}(E)$ as a $K$-base, and is a $\mathbb{Z}$-graded subalgebra of the Leavitt path algebra $L_K(E)$ (see, e.g., \cite[Lemma 1.6]{G}). This observation, together with Lemma \ref{lem:graphaut}, shows that every automorphism $\sigma$ of $E$ induces a $\mathbb{N}$-graded automorphism of $KE$. Moreover, as a consequence  of Theorem \ref{thm:twist-is-LPA}, we obtain the following fact.

\begin{proposition}\label{prop:path-algebra-twist}
Let $K$ be a field, $E$ be a graph, and $\sigma$ an automorphism of $E$, with induced graded automorphism of $KE$ also written $\sigma$. Then
\[
(KE)^{\sigma}\;\cong\;KE_{\sigma}
\]
as $\mathbb N$-graded $K$-algebras.
\end{proposition}
\begin{proof}
By the proof of Theorem \ref{thm:twist-is-LPA}, there is a unique $\mathbb{Z}$-graded $K$-algebra isomorphism $$\Theta\colon L_\K(E_{\sigma})\longrightarrow L_K(E)^{\sigma}$$ such that 
\begin{center}
$\Theta(v)=v$,\quad $\Theta(e)=e$,\quad and \quad $\Theta(e^*)= (\sigma_1^{-1}(e))^*$   
\end{center}
for all $v\in E^0$ and $e\in E^1$. Moreover, as shown in the proof of Theorem \ref{thm:twist-is-LPA}, for every path $\alpha$ in $E$, there exists a path $p$ in $E_{\sigma}$ such that $\Theta(p) = \alpha$. Since $\text{Path}(E_{\sigma})$ and 
$\text{Path}(E)$ are $K$-bases of $KE_{\sigma}$ and $KE$, respectively, these observations show that the restriction $$\Theta|_{KE_{\sigma}}: KE_{\sigma} \longrightarrow (KE)^{\sigma}$$ is an $\mathbb{N}$-graded $K$-algebra isomorphism, thus completing the proof.
\end{proof}

Combining these observations with $(4)$ yields an independent derivation of the categorical consequence of Theorem \ref{thm:twist-is-LPA}, without any appeal to Zhang's theorem as established in \cite{Zhang96}.

\begin{corollary}\label{cor:smith-dictionary}
Let $K$ be a field, $E$ a finite graph with no sinks and no sources, and $\sigma$ an automorphism of $E$. Then    
\[
\mathsf{Gr}\text{-}L_K(E)\;\equiv\;\mathsf{QGr}\text{-}(KE)\equiv\;\mathsf{QGr}\text{-}KE_{\sigma}
\;\equiv\;\mathsf{Gr}\text{-}L_K(E_{\sigma}).
\]
\end{corollary}
\begin{proof}
Since $E$ is a finite graph with no sinks and no sources, it follows that $E^{\circ}=E$ and, by Lemma
\ref{lem:circ-commutes}, that $(E_\sigma)^{\circ}=E_\sigma$. 
Let $m$ be the order of $\sigma_0$.
By Proposition \ref{prop:adjacency}, the adjacency matrices of the graphs $E^{(m)}$ and $(E_\sigma)^{(m)}$
are $A^{m}_E$ and $A_{E_\sigma}^{\,m}=A^{m}_E$, respectively.
Since a graph is determined up to isomorphism by its vertex set and adjacency matrix, it follows that
$$E^{(m)}\cong(E_\sigma)^{(m)}.$$
Applying $(4)$ twice then yields
\[
\mathsf{Gr}\text{-}L_K(E)\;\equiv\;\mathsf{QGr}\text{-}(KE)\;\equiv\;\mathsf{QGr}\text{-}KE^{(m)}
\;=\;\mathsf{QGr}\text{-}K(E_\sigma)^{(m)}\;\equiv\;\mathsf{QGr}\text{-}KE_{\sigma}
\;\equiv\;\mathsf{Gr}\text{-}L_K(E_{\sigma}),
\] thus finishing the proof. 
\end{proof}

\section{Invariance of ring-theoretic properties of Leavitt path algebras under twisting}\label{Invariant}

\noindent Let $E=(E^0,E^1,s,r)$ be a graph and $\sigma=(\sigma_0,\sigma_1)$  an automorphism of $E$. Since $E$ and $E_{\sigma}$ have the \emph{same} source map, $s^{-1}(v)$ is literally the same set of edges in both graphs. In particular $E$ and $E_{\sigma}$ have the same sinks, the same infinite emitters and the same regular vertices, and $|s^{-1}(v)|$ is the same  for every $v\in E^0$. What the twist changes is the range map and, consequently, the cycles. The first purpose of this section is to describe precisely how the cycles are affected.

For any positive integer $k$, Lemma \ref{lm:twsist-path} provides a map $$\Phi_k: \text{Path}(E_{\sigma})\longrightarrow \text{Path}(E)$$ defined by
\[
\Phi_k(e_1e_2\cdots e_k):= f_1f_2\cdots f_k,
\] where $f_j:=\sigma_1^{\,j-1}(e_j)$ for all $1\le j\le k$.
We let $m$ denote the order of $\sigma_0$ as a permutation of $E^0$; in particular, if $E^0$ is finite, then $m<\infty$.

\begin{lemma}\label{lem:path-endpoints}
For every $k\ge 1$ the map $\Phi_k$ is a bijection from the set of paths of length $k$
in $E_{\sigma}$ onto the set of paths of length $k$ in $E$, and for a path
$p=e_1\cdots e_k$ in $E_{\sigma}$ one has
\[
s\bigl(\Phi_k(p)\bigr)=s(p),
\qquad
r\bigl(\Phi_k(p)\bigr)=\sigma_0^{\,k}\bigl(r_\sigma(p)\bigr).
\]
\end{lemma}

\begin{proof}
By Lemma \ref{lm:twsist-path}, $\Phi_k$ is a bijection from the set of paths of length $k$ in $E_{\sigma}$ onto the set of paths of length $k$ in $E$. We now determine its effect on the endpoints. Since $f_1=e_1$, it follows that
$$s(\Phi_k(p))=s(e_1)=s(p).$$ Moreover, 
\[
r\bigl(\Phi_k(p)\bigr)=r\bigl(\sigma_1^{\,k-1}(e_k)\bigr)=\sigma_0^{\,k-1}\bigl(r(e_k)\bigr)
=\sigma_0^{\,k-1}\sigma_0\bigl(r_\sigma(e_k)\bigr)=\sigma_0^{\,k}\bigl(r_\sigma(p)\bigr),
\]
where we have used $r\circ\sigma_1=\sigma_0\circ r$ repeatedly, together with $r=\sigma_0\circ r_\sigma$.
\end{proof}

Specializing to closed paths, we see that a closed path in 
$E_\sigma$ need not be a closed path in $E$; rather, it becomes a \emph{$\sigma_0$-twisted} closed path.

\begin{proposition}\label{prop:closed-paths}
Let $v\in E^0$ and $k\ge 1$. Then $\Phi_k$ restricts to a bijection
\[
\bigl\{\text{closed paths of length $k$ based at $v$ in }E_\sigma\bigr\}
\;\longleftrightarrow\;
\bigl\{\text{paths of length $k$ from $v$ to $\sigma_0^{\,k}(v)$ in }E\bigr\}.
\]
In particular, if the order of $\sigma_0$ is $m$ and $m\mid k$, then $\Phi_k$ maps the set of closed paths of length $k$ based at
$v$ in $E_\sigma$ bijectively onto the set of closed paths of length $k$ based at $v$ in $E$.
\end{proposition}

\begin{proof}
It follows immediately from Lemma~\ref{lem:path-endpoints} that $s(p)=v=r_\sigma(p)$ holds if and only if
\begin{center}
$s(\Phi_k(p))=v$ and $r(\Phi_k(p))=\sigma_0^{\,k}(v)$.    
\end{center}
The final assertion follows when
$\sigma_0^{\,k}=\mathrm{id}$, thus finishing the proof.
\end{proof}

By contrast, neither the number of cycles nor their individual lengths is preserved. Example~\ref{ex:C2} already shows that a single cycle of length $2$ can become two loops.

\begin{example}\label{ex:Cm-rotation}
Let $C_m$ be the graph of Example~\ref{C_m} with vertices $v_1,\dots ,v_m$ and
edges $e_i\colon v_i\to v_{i+1}$, where the indices are taken modulo $m$, and let $\sigma$ be the rotation
$v_i\mapsto v_{i+1}$, $e_i\mapsto e_{i+1}$, so that $\sigma_0$ has order exactly $m$. Then
\[
r_\sigma(e_i)=\sigma_0^{-1}(v_{i+1})=v_i=s(e_i),
\]
so $(C_m)_\sigma=\bigsqcup_{i=1}^{m}R_1$ and, by Theorem~\ref{thm:twist-is-LPA},
\[
L_K(C_m)^\sigma\cong K[x,x^{-1}]^{\,m},
\qquad\text{whereas}\qquad
L_K(C_m)\cong M_m\bigl(K[x,x^{-1}]\bigr).
\]
On the level of adjacency matrices, 
 $A_{C_m}=P$ is the cyclic permutation matrix, while
$A_{(C_m)_\sigma}=PP^{-1}=I$. In agreement with Proposition~\ref{prop:adjacency},
$A_{(C_m)_\sigma}^{\,m}=I=A_{C_m}^{\,m}$. The single cycle of length $m$ in $C_m$ has been
replaced by $m$ loops, but the $m$-th powers of the adjacency matrices, and hence the
counts of closed paths of length divisible by $m$ agree, in accordance with
Proposition~\ref{prop:closed-paths}. The case $m=2$ is Example~\ref{ex:C2}.
\end{example}

\begin{lemma}\label{lem:symmetry}
Let $E$ be a graph and $\sigma$ an automorphism of $E$. Then
$\sigma$ is an automorphism of the graph $E_{\sigma}$, and $(E_{\sigma})_{\sigma^{-1}}=E$.
\end{lemma}
\begin{proof}
The source maps of $E$ and $E_\sigma$ agree, so $s\circ\sigma_1=\sigma_0\circ s$ still
holds. Moreover $r_\sigma\circ\sigma_1=\sigma_0^{-1}\circ r\circ\sigma_1=r$ and
$\sigma_0\circ r_\sigma=r$, whence $r_\sigma\circ\sigma_1=\sigma_0\circ r_\sigma$. Finally
the range map of $(E_\sigma)_{\sigma^{-1}}$ is
$(\sigma_0^{-1})^{-1}\circ r_\sigma=\sigma_0\sigma_0^{-1}\circ r=r$, thus completing the proof.
\end{proof}

A graph $E$ satisfies {\it Condition $(\mathrm{K})$} if for each $v\in E^0$ which lies on a closed simple path, there exist at least two distinct closed simple paths $\alpha, \beta$ based at $v$. A graph $E$ is called a {\it no-exit graph} if no cycle in $E$ has an exit.

Using Proposition \ref{prop:closed-paths} and Lemma \ref{lem:symmetry}, we obtain the following useful corollary.

\begin{corollary}\label{cor:acyclic}
Let $E$ be a graph such that $E^0$ is finite, and let $\sigma$ be an automorphism of $E$. Then
\begin{enumerate}
\item[$(1)$] $E$ is acyclic if and only if $E_{\sigma}$ is acyclic;
\item[$(2)$] Every cycle in $E$ has an exit if and only if  every cycle in $E_{\sigma}$ has an exit;
\item[$(3)$] $E$ is a no-exit graph if and only if   $E_{\sigma}$ is a no-exit graph;
\item[$(4)$] $E$ satisfies Condition~$(\mathrm{K})$ if and only if $E_\sigma$ satisfies Condition~$(\mathrm{K})$.
\end{enumerate}
\end{corollary}

\begin{proof}
We denote $m$ the order of $\sigma_0$. By Lemma~\ref{lem:symmetry}, it suffices in $(1)$, $(2)$ and $(3)$ to prove one implication in each case.

$(1)$ Let $c$ be a cycle of length $k$ in $E_{\sigma}$ based at $v$. Then $c^m$ is a closed
path of length $km$ based at $v$ in $E_{\sigma}$. Since $m\mid km$, Proposition~\ref{prop:closed-paths} implies that
$\Phi_{km}(c^m)$ is a closed path of length
$km\ge 1$ based at $v$ in $E$, and so $E$ contains a cycle.

$(2)$ Suppose $E_{\sigma}$ has a cycle $c= e_1\cdots e_k$ without exits, based at $v$. Then
$\bigl|s^{-1}(s(e_j))\bigr|=1$ for all $1\le j\le k$. Since $E$ and $E_\sigma$ have the same
source map, the same equality holds in $E$. Since $\sigma$ is an automorphism of $E$, the function $u\mapsto |s^{-1}(u)|$ is $\sigma_0$-invariant. Write
$\Phi_{km}(c^m)=f_1\cdots f_{km}$. By construction, $s(f_j)=\sigma_0^{\,j-1}(s(e_{j'}))$,
where $j'\in\{1,\dots ,k\}$ is the residue of $j$ modulo $k$. Hence, every vertex on
$\Phi_{km}(c^m)$ emits exactly one edge in $E$. As in $(1)$, $\Phi_{km}(c^m)$ is a closed path in $E$. From these observations, we obtain that $\Phi_{km}(c^m)$ is a power of a cycle all of whose vertices emit exactly one edge, i.e. of a cycle without exits. Therefore, $E$ contains a cycle without exits.

$(3)$ Assume that $E$ is a no-exit graph and $c= e_1\cdots e_k$ is a cycle of length $k$ in $E_{\sigma}$. By Proposition~\ref{prop:closed-paths}, $\Phi_{km}(c^m) = f_1\cdots f_{km}$ is a closed path in $E$. Since $E$ is a no-exit graph, $\Phi_{km}(c^m)$ is exactly a power of a cycle without exits. By construction, $s(f_j)=\sigma_0^{\,j-1}(s(e_{j'}))$,
where $j'\in\{1,\dots ,k\}$ is the residue of $j$ modulo $k$. 
Since $\sigma$ is an automorphism of $E$, the function $u\mapsto |s^{-1}(u)|$ is $\sigma_0$-invariant. Then, since $E$ and $E_\sigma$ have the same source map, it follows that
$$\bigl|s^{-1}(s(e_j))\bigr|= \bigl|s^{-1}(s(f_j))\bigr|=1$$ for all $1\le j\le k$. Therefore, $c$ has no exits, and so $E_{\sigma}$ is a no-exit graph.

$(4)$ Assume that $E$ satisfies Condition $(\mathrm{K})$, whereas $E_{\sigma}$ does not satisty Condition $(\mathrm{K})$.
Then there exist a vertex $v$ and a unique simple closed path $c$ in $E_{\sigma}$ based at $v$. By Proposition~\ref{prop:closed-paths}, $\Phi_{km}(c^m)$ is a closed path of length $km$ based at $v$ in $E$, where $ k := |c|$. Since $E$ satisfies Condition $(\mathrm{K})$, there exist distinct simple closed paths $d_1$ and $d_2$ based at $v$ in $E$. By Proposition~\ref{prop:closed-paths} again, $\Phi^{-1}_{kmt_1}(d_1^{km})$ and  $\Phi^{-1}_{kmt_2}(d_2^{km})$ are closed paths based at $v$ in $E_{\sigma}$, where $t_i := |d_i|$ for $i=1, 2$. Since $c$ is the unique simple closed path based at $v$ in $E_{\sigma}$, it follows that 
\begin{center}
$\Phi^{-1}_{kmt_1}(d_1^{km}) = c^{mt_1}$ and $\Phi^{-1}_{kmt_2}(d_2^{km}) = c^{mt_2}$.   
\end{center}
Equivalently, 
\begin{center}
$\Phi_{kmt_1}(c^{mt_1}) = d_1^{km}$ and $\Phi_{kmt_2}(c^{mt_2}) = d_2^{km}$.   
\end{center}
By Proposition~\ref{prop:closed-paths}, $\Phi_{kmt_1t_2}(c^{mt_1t_2})$ is a closed path in $E$. Moreover, the above equalities imply that
$$d_1^{km}\alpha = \Phi_{kmt_1t_2}(c^{mt_1t_2}) = d_2^{km}\beta$$ for some paths $\alpha$ and $\beta$ in $E$. This is impossible, since 
 $d_1$ and $d_2$ are distinct simple closed paths in $E$, thus completing the proof.
\end{proof}

Consequently, we obtain several properties of Leavitt path algebras that are invariant under the twists.

\begin{proposition}\label{prop:strongly-graded-preserved}
Let $K$ be a field, $E$ a graph such that $E^0$ is finite, and $\sigma$ an automorphism of $E$. Then the following statements hold:

$(1)$ $L_K(E)$ is strongly graded if and only if $L_K(E)^{\sigma}$ is strongly graded;

$(2)$ Every primitive idempotent in $L_K(E)$ is minimal if and only if the same holds for $L_K(E)^{\sigma}$;

$(3)$ $L_K(E)$ is von Neumann regular if and only if $L_K(E)^{\sigma}$ is von Neumann regular;

$(4)$ $L_K(E)$ is Dedekind finite if and only if $L_K(E)^{\sigma}$ is Dedekind finite;

$(5)$ $L_K(E)$ is stably finite if and only if $L_K(E)^{\sigma}$ is stably finite;

$(6)$ $L_K(E)$ is an exchange ring if and only if $L_K(E)^{\sigma}$ is an exchange ring.
\end{proposition}
\begin{proof}
(1) The graphs $E$ and $E_{\sigma}$ have the same source map and, consequently, the same sinks. By \cite[Theorem 4.2]{CHR}, the unital Leavitt path algebra $L_K(F)$ of a graph $F$ is strongly graded if and only if $F$ is row-finite and has no sinks. Hence, the desired statement immediately follows from these observations and Theorem~\ref{thm:twist-is-LPA}.

(2) It is well known (see, e.g., \cite[Theorem 3.5.7]{AAS}) that for every graph $F$, every primitive idempotent in $L_K(F)$ is minimal if and only if every cycle in $F$ has an exit.
Using this result together with Theorem~\ref{thm:twist-is-LPA} and Corollary~\ref{cor:acyclic}(2), we immediately obtain the desired statement.

(3) It is well known (see, e.g., \cite[Theorem 3.4.1]{AAS}) that the Leavitt path algebra of an arbitrary $F$ over $K$ is von Neumann regular if and only if $F$ is acyclic. Thus, using this result together with Theorem~\ref{thm:twist-is-LPA} and Corollary~\ref{cor:acyclic}(1), we immediately obtain the desired statement.

By \cite[Theorem 4.13]{Ben}, a unital Leavitt path algebra $L_K(F)$ of a graph $F$ is stably finite if and only if $L_K(F)$ is Dedekind finite if aond only if $F$ is a no-exit graph. Thus, using these results together with Theorem~\ref{thm:twist-is-LPA} and Corollary~\ref{cor:acyclic}(3), we immediately obtain the desired statements (4) and (5).

(6) It is well known (see, e.g., \cite[Theorem 3.3.11]{AAS}) that the Leavitt path algebra of an arbitrary $F$ over $K$ is an exchange ring if and only if $F$ satisfies Condition $(\mathrm{K})$. Thus, using this result together with Theorem~\ref{thm:twist-is-LPA} and Corollary~\ref{cor:acyclic}(4), we immediately obtain the desired statement.
\end{proof}

It is worth mentioning the following example.

\begin{example}\label{ex:inf-graph}
 Let $E$ be the graph with 
 \begin{center}
$E^0= \{v_n\mid n\in \mathbb{Z}\}$ and $E^1 = \{e_n\mid n\in \mathbb{Z}\},$     
 \end{center}
where $s(e_n) = v_n$ and $r(e_n) = v_{n+1}$ for all $n$. Hence, $E$ is the following graph:
$$\xymatrix{... \ar[r]^{e_{-2}}&v_{-1} \ar[r]^{e_{-1}}&v_0 \ar[r]^{e_{0}}&v_1 \ar[r]^{e_1}&v_2 \ar[r]^{e_{2}}&...}$$ 
Let $\sigma = (\sigma_0, \sigma_1)$ be the automorphism of $E$ defined by
 \begin{center}
  $\sigma_0(v_n) = v_{n+1}$ and   $\sigma_1(e_n) = e_{n+1}$   
 \end{center}
for all $n$. We then have that $E_{\sigma}$ is the union of $\mathbb{Z}$-indexed copies of the graph $R_1$. In other words, $E_{\sigma}$ is the following graph:
$$
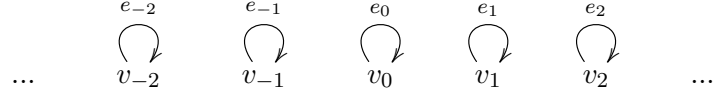
\xymatrix{... &v_{-2}\ar@(ul,ur)^{e_{-2}} &v_{-1}\ar@(ul,ur)^{e_{-1}} &v_0\ar@(ul,ur)^{e_0} &v_1\ar@(ul,ur)^{e_1} & v_2\ar@(ul,ur)^{e_2}&  ...}$$ 
Moreover, statements (1), (2) and (4) of Corollary \ref{cor:acyclic}, as well as statements (2), (3) and (6) of Proposition \ref{prop:strongly-graded-preserved} do not hold for $E$.
\end{example}

A ring $R$ is said to have {\it Unbounded Generating Number} ({\it UGN} for short) if, for each positive integer $m$, any set of generators for the free right $R$-module $R^m$ has cardinality $\ge m$. Equivalently, if there is an epimorphism of right free $R$-modules $R^n \longrightarrow R^m$, then $n\ge m$ (see, e.g., \cite[Remark 2.2]{ANP}). For Leavitt path algebra $L_K(E)$ of a finite graph $E$, the authors of \cite[Theorem 3.16]{ANP} gave criteria for when $L_K(E)$ has UGN. Note that $L_K(E)$ has UGN if and only if $L_K(E)$ is algebraically amenable (see \cite[Remark 3.17]{ANP} and \cite[Corollary 5.11]{ALLW}), and if and only if $L_K(E)$ has a nonzero finite-dimensional right module (see \cite[Corollary 6.7]{KO}). 

The next goal of this section is to show that the UGN property of a Leavitt path algebra is invariant under twists. To do so, we need some useful notions and facts. First, a cycle $c=e_1\cdots e_k$ in a graph $E$ is called a {\it source cycle} if $|r^{-1}(s(e_i))| =1$ for all $1\le i\le k$.

\begin{lemma}\label{lm:source-cycle}
Let $E$ be a graph such that $E^0$ is finite and let $\sigma$ be an automorphism of $E$. Then $E$ has a source cycle if and only if $E_{\sigma}$ has a source cycle.
\end{lemma}
\begin{proof} We denote $m$ the order of $\sigma_0$. By Lemma~\ref{lem:symmetry}, it suffices to prove one implication.
Assume that $E_{\sigma}$ contains a source cycle $c= e_1\cdots e_k$. We then have that $|r^{-1}_{\sigma}(s(e_i))| = 1$ for all $1\le i\le k$. We note that
$$f\in  r^{-1}_{\sigma}(s(e_i)) \Longleftrightarrow r(f) = \sigma_0(s(e_i)) \Longleftrightarrow f \in r^{-1}(\sigma_0(s(e_i))),$$ and so $|r^{-1}_{\sigma}(s(e_i))| = |r^{-1}(\sigma_0(s(e_i)))| = 1$ for all $1\le i\le k$. Since $\sigma$ is an automorphism of $E$, it follows that
$$|r^{-1}(\sigma^t_0(s(e_i)))| = |r^{-1}(\sigma_0(s(e_i)))| = 1$$ for all $t$ and $i$. By Proposition \ref{prop:closed-paths}, $\Phi_{km}(c^m) = f_1\cdots f_{km}$ is a closed path in $E$, where $s(f_j)=\sigma_0^{\,j-1}(s(e_{j'}))$,
where $j'\in\{1,\dots ,k\}$ is the residue of $j$ modulo $k$. Moreover, \[|r^{-1}(s(f_j))| = |r^{-1}(\sigma_0^{j-1}(s(e_{j'})))| =1\] for all $1\le j \le km$. Hence, $\Phi_{km}(c^m)$ is exactly a power of a source cycle. Thus, $E$
 contains a source cycle, completing the proof.
\end{proof}

Let $E = (E^0, E^1, r, s)$ be a finite graph containing a cycle, and $\sigma$ an automorphism of $E$. We denote by $E_{sf}$ the graph without sources obtained from $E$ by repeatedly deleting all sources. Note that $E_{sf}$ has a cycle and is independent of the source elimination process (see \cite[Lemma 3.13]{ANP}). We denote by $S(E)$ the set of all sources in $E$. It is clear that $S(E)$ is $\sigma_0$-invariant. We define a new graph $E\smallsetminus_{S(E)}$ by
\begin{center}
$(E\smallsetminus_{S(E)})^0 = E^0\setminus S(E)$    and  $(E\smallsetminus_{S(E)})^1 = E^1\setminus s^{-1}(S(E))$
\end{center}
with source and range maps given by the restrictions of 
$s$ and $r$, respectively. Since $S(E)$ is $\sigma_0$-invariant, $\sigma$ induces an automorphism of $E\smallsetminus_{S(E)}$. Moreover, $v\in S(E)$ is an isolated vertex (i.e., $v$ is both a source and a sink) if and only if $\sigma_0(v)$ is an isolated vertex of $S(E_{\sigma})$. Using these notes together with \cite[Theorem 3.16]{ANP}, we obtain the following.

\begin{proposition}\label{prop:UGN}
Let $K$ be a field, $E$ a finite graph, and $\sigma$ an automorphism of $E$. Then $L_K(E)$ has Unbounded Generating Number (UGN, for short) if and only if $L_K(E)^{\sigma}$ has Unbounded Generating Number.   
\end{proposition}
\begin{proof}
By Lemma~\ref{lem:symmetry}, it suffices to prove one implication. Assume that  $L_K(E)$ has UGN. We claim that $L_K(E_{\sigma})$ has UGN. Indeed, if $E$ is acyclic, then by Corollary \ref{cor:acyclic}(1), $E_{\sigma}$ is an acyclic graph. By \cite[Theorem 3.16]{ANP}, $L_K(E_{\sigma})$ has UGN. Consider the case when $E$ has a cycle. By Corollary \ref{cor:acyclic}(1), $E_{\sigma}$ has a cycle. We define a sequence of finite graphs $(E_n)_{n\in \mathbb{N}}$  recursively by
\begin{center}
$E_0 := E$\quad and\quad $E_{n+1} := E_n\smallsetminus_{S(E_n)}$    
\end{center}
for all $n$. Since $E$ is finite, there exists a smallest positive integer $n_0$ such that $E_{n_0} = E_{sf}$. By induction on $n_0$, we obtain that $S(E_i) = S((E_{\sigma})_i)$, $\sigma$ induces an automorphism of $E_i$, and $(E_i)_{\sigma} = (E_{\sigma})_i$ for every $1\le i\le n_0$. In particular, $(E_{sf})_{\sigma} = (E_{\sigma})_{sf}$.

If there exists $k$ with $0\le k < n_0$ such that $S(E_k)$ contains an isolated vertex $v$, then $S((E_{\sigma})_k)$ contains the isolated vertex $\sigma_0(v)$. By \cite[Theorem 3.16]{ANP}, $L_K(E_{\sigma})$ has Unbounded Generating Number.
Otherwise, $S(E_i)$ contains no isolated vertex for every $0\le i< n_0$. Then the same holds for each $S((E_{\sigma})_i)$ for every $0\le i< n_0$. Since $L_K(E)$ has UGN, it follows from \cite[Theorem 3.16]{ANP} that $E_{sf}$ contains a source cycle. By Lemma \ref{lm:source-cycle}, $(E_{\sigma})_{sf}$ also contains a source cycle. Applying \cite[Theorem 3.16]{ANP} again, we conclude that $L_K(E_{\sigma})$ has UGN, proving the claim.  By Theorem \ref{thm:twist-is-LPA},  $L_K(E)^{\sigma}$ has UGN, thus completing the proof.  
\end{proof}

We say that a $\mathbb Z$-graded algebra $A$ has \emph{gr-IBN} if every graded right $A$-module isomorphism $A^{n}\cong A^{m}$ implies $m=n$.

\begin{proposition}\label{prop:gr-IBN}
Let $A$ be a $\mathbb Z$-graded $K$-algebra and let $\tau$ be a twisting system. Then $A$ has gr-IBN if and only if $A^{\tau}$ has gr-IBN, and the two algebras have the same graded module type. Moreover, IBN implies gr-IBN.
\end{proposition}

\begin{proof}
Zhang's twisting functor $F\colon\mathsf{Gr}\text{-}A\to\mathsf{Gr}\text{-}A^{\tau}$ is an
additive equivalence with $F(A_A)\cong A^{\tau}_{A^{\tau}}$ \cite{Zhang96}, so
$F(A^{n})\cong (A^{\tau})^{n}$ and $A^{n}\cong A^{m}$ in $\mathsf{Gr}\text{-}A$ if and only
if $(A^{\tau})^{n}\cong(A^{\tau})^{m}$ in $\mathsf{Gr}\text{-}A^{\tau}$. The last assertion
is immediate as a graded isomorphism $A^{n}\cong A^{m}$ is, in particular, an ungraded one.
\end{proof}

The twist by an automorphism of finite order can be realized inside a skew group algebra, which is free of finite rank over the original algebra. This yields the following answer to Question \ref{q:IBN} in that case.

\begin{theorem}\label{thm:IBN-finite-order}
Let $A=\bigoplus_{n\in\Z}A_n$ be a unital $\Z$-graded $K$-algebra and let $\sigma$ be a graded automorphism of $A$ of finite order $N$. Then the following statements hold:
\begin{enumerate}
\item $A$ has IBN if and only if $A^\sigma$ has IBN;
\item $A$ has UGN if and only if $A^\sigma$ has UGN.
\end{enumerate}
\end{theorem}

\begin{proof}
For $n\in\Z$ write $\bar n$ for its residue in $\Z/N\Z$; since $\sigma^N=\mathrm{id}$, the
automorphism $\sigma^{\bar n}:=\sigma^n$ is well defined. Let $C_N=\langle g\mid g^N=1\rangle$ denote the cyclic group of order $N$. Since
$\sigma^N=\mathrm{id}$, the assignment $g\cdot a=\sigma(a)$ defines an action of $C_N$ on $A$
by graded algebra automorphisms.

Let $S=A\rtimes_\sigma C_N$ be the skew
group algebra, that is,
\[
S=\bigoplus_{j\in\Z/N\Z}A\,g^j,\qquad (a g^i)(b g^j)=a\,\sigma^i(b)\,g^{i+j}.
\]
This product is well defined and associative because $\sigma^N=\mathrm{id}$, and $g^0=1$ is the
identity. We identify $A$ with the subalgebra $Ag^0$ of $S$. Since $g^jb=\sigma^j(b)g^j$ and
$\sigma^j$ is bijective, we have $Ag^j=g^jA$, and $b\mapsto g^jb$ is an isomorphism of right
$A$-modules $A_A\to (g^jA)_A$. Hence $S_A=\bigoplus_{j}g^jA$ is a free right $A$-module of
rank~$N$.

Define a $K$-linear map $\iota\colon A^\sigma\to S$ by $\iota(a)=a\,g^{\bar n}$ for
$a\in A_n$. For $a\in A_n$ and $b\in A_k$ we have $a\star b=a\sigma^n(b)\in A_{n+k}$, and so
\[
\iota(a)\iota(b)=(a g^{\bar n})(b g^{\bar k})=a\,\sigma^n(b)\,g^{\overline{n+k}}
=\iota(a\star b).
\]
Moreover $\iota(1)=1$, so $\iota$ is a unital $K$-algebra homomorphism. It is injective: if
$\iota\bigl(\sum_n a_n\bigr)=0$ with $a_n\in A_n$, then $\sum_{n\equiv j}a_n=0$ in $A$ for each
$j\in\Z/N\Z$, and since the $a_n$ lie in distinct homogeneous components, every $a_n$ is zero.

Regard $S$ as a left $A^\sigma$-module via $\iota$. Let $\varphi\colon (A^\sigma)^p\to
(A^\sigma)^q$ be a surjective homomorphism of right $A^\sigma$-modules. By right exactness of
the tensor product, $\varphi\otimes_{A^\sigma}S\colon S^p\to S^q$ is a surjective homomorphism of
right $S$-modules, and it is an isomorphism if $\varphi$ is. Restricting scalars along
$A\subseteq S$ and using $S_A\cong A^N$, we obtain a surjective homomorphism of right
$A$-modules $A^{Np}\to A^{Nq}$, which is an isomorphism if $\varphi$ is. Consequently, if $A$ has
IBN (respectively UGN), then $Np=Nq$ (respectively $Np\ge Nq$), so $p=q$ (respectively $p\ge q$).
This shows that $A^\sigma$ has IBN (respectively UGN) whenever $A$ does.

For the converse, note that $\sigma$ is a graded automorphism of $A^\sigma$: for $a\in A_n$ and
$b\in A$,
\[
\sigma(a\star b)=\sigma(a)\,\sigma^{n+1}(b)=\sigma(a)\,\sigma^n\bigl(\sigma(b)\bigr)
=\sigma(a)\star\sigma(b).
\]
Hence $\sigma^{-1}$ is a graded automorphism of $A^\sigma$ of order $N$, and for $a\in A_n$ we
have $a\star\sigma^{-n}(b)=a\,\sigma^n\sigma^{-n}(b)=ab$, that is,
$(A^\sigma)^{\sigma^{-1}}=A$. Applying the first part to $A^\sigma$ and $\sigma^{-1}$ completes
the proof.
\end{proof}

\begin{corollary}\label{cor:IBN-graphs}
Let $K$ be a field, $E$ a finite graph, and $\sigma$ an automorphism of $E$. Then $\LE$ has
$\IBN$ if and only if $\LE^\sigma$ has $\IBN$, if and only if $L_K(E_{\sigma})$ has $\IBN$.
\end{corollary}

\begin{proof}
Since $E$ is finite, its automorphism group is finite, so $\sigma=(\sigma_0,\sigma_1)$ has finite order $N$. As $\LE$ is generated by $E^0\cup E^1\cup\{e^*\mid e\in E^1\}$, the induced graded automorphism of $\LE$ (Lemma \ref{lem:graphaut}) also has order $N$. The first equivalence now follows from Theorem~\ref{thm:IBN-finite-order}, and the second from Theorem \ref{thm:twist-is-LPA}.
\end{proof}

\begin{remark}\label{rem:order}
(1) Theorem~\ref{thm:IBN-finite-order}(2) together with Theorem \ref{thm:twist-is-LPA} gives a second proof of
Proposition \ref{prop:UGN}, which does not use the characterization of UGN in \cite[Theorem 3.16]{ANP}. Note also that the proof of
Theorem \ref{thm:IBN-finite-order} only uses that $\iota$ is a unital ring homomorphism, not that it is injective.

(2) The order $N$ of $\sigma=(\sigma_0,\sigma_1)$ may be strictly larger than the order $m$ of
$\sigma_0$ used elsewhere in this section. For instance, if $E=R_2$ and $\sigma_1$ interchanges
the two loops, then $m=1$ while $N=2$.

(3) The finiteness of the order of $\sigma$ is used in Theorem~\ref{thm:IBN-finite-order} only to make $S$ free of
finite rank over $A$. Let $\sigma$ be an arbitrary graded automorphism of $A$, and let
$B=A[t,t^{-1};\sigma]$ be the skew Laurent polynomial ring, in which $ta=\sigma(a)t$ for
$a\in A$. As in the proof of Theorem~\ref{thm:IBN-finite-order}, $\sigma$ is a graded automorphism of $A^\sigma$, and
the map
\[
A^\sigma[s,s^{-1};\sigma]\longrightarrow B,\qquad as^k\longmapsto a\,t^{n+k}\quad(a\in A_n),
\]
is an isomorphism of rings. Indeed, it is bijective on the bases $\{A_ns^k\}$ and
$\{A_nt^{n+k}\}$, and for $a\in A_n$ and $b\in A_m$, both $(as^k)(bs^l)=a\,\sigma^{n+k}(b)\,s^{k+l}$
and $(at^{n+k})(bt^{m+l})$ give $a\,\sigma^{n+k}(b)\,t^{n+k+m+l}$. Thus $A$ and $A^\sigma$
are both unital subrings of $B$. Since a ring admitting a unital homomorphism into a ring with
$\IBN$ has $\IBN$, both $A$ and $A^\sigma$ have $\IBN$ whenever $B$ does. Consequently, a counterexample to
Question \ref{q:IBN} given by a single graded automorphism $\sigma$ requires $A$ to have $\IBN$ while
$A[t,t^{-1};\sigma]$ does not. When $\sigma^N=\mathrm{id}$, the element $t^N$ is central in $B$,
and $B/(t^N-1)$ is the skew group algebra $S$ of Theorem~\ref{thm:IBN-finite-order}.
\end{remark}

\begin{corollary}\label{cor:counterexample-requirements}
Let $A$ be a $\Z$-graded algebra with $\IBN$ and $\tau$ a twisting system such that $A^\tau$ does not
have $\IBN$. Then $A^\tau$ has gr-$\IBN$, and $\tau$ is not of the form $\{\sigma^n\mid n\in\Z\}$ for a
graded automorphism $\sigma$ of $A$ of finite order. In particular, no counterexample to
Question \ref{q:IBN} exists among algebras for which gr-$\IBN$ implies $\IBN$, nor among twists by graded
automorphisms of finite order; in particular, none arises from Theorem \ref{thm:twist-is-LPA} applied to a finite
graph.
\end{corollary}

\begin{proof}
The first assertion is Proposition \ref{prop:gr-IBN}, and the second is Theorem~\ref{thm:IBN-finite-order}.
The last assertion follows from Corollary~\ref{cor:IBN-graphs}.
\end{proof}

Recall the graph monoid $M_E$ of a finite graph $E$, the commutative monoid with generators $\{v\mid v\in
E^0\}$ and relations $v=\sum_{e\in s^{-1}(v)}r(e)$ for regular $v$; by \cite{AMP} one has
$\mathcal V\bigl(L_K(E)\bigr)\cong M_E$ with $[L_K(E)]=\sum_{v\in E^0}v$, so that
$L_K(E)^{n}\cong L_K(E)^{m}$ if and only if $n\cdot\mathbf 1=m\cdot\mathbf 1$ in $M_E$,
where $\mathbf 1:=\sum_{v\in E^0}v$.

Since $M_E$ depends only on the adjacency matrix $A=A_E$, we also write $M_A$ for $M_E$. For a
commutative monoid $M$ and $x\in M$, we say that $(M,x)$ has IBN if $nx=mx$ implies $n=m$;
otherwise the \emph{module type} of $(M,x)$ is the pair $(m,n)$ with $m<n$, $mx=nx$, $m$ minimal
and, subject to this, $n$ minimal. Thus the module type of $\LE$ is that of $(M_E,1)$.

\begin{corollary}\label{cor:module-type}
Let $E$ be a finite graph and $\sigma$ an automorphism of $E$, let $A=A_E$, and let $P$ be the
permutation matrix of $\sigma_0$. Then the following statements hold:
\begin{enumerate}
\item The module types of $\LE$ and $\LE^\sigma$ are those of $(M_A,1)$ and
$(M_{AP^{-1}},1)$, respectively.
\item $(M_A,1)$ has IBN if and only if $(M_{AP^{-1}},1)$ has IBN.
\item The module types $\LE$ and $\LE^\sigma$ need not coincide.
\item Conversely, every pair $(A,P)$ consisting of a nonnegative integer square matrix $A$ and a
permutation matrix $P$ with $PAP^{-1}=A$ arises from such a pair $(E,\sigma)$.
\end{enumerate}
\end{corollary}

\begin{proof}
(1) This follows from the description of $\mathcal V(\LE)$ recalled above, together with
Theorem \ref{thm:twist-is-LPA} and the equality $A_{E_\sigma}=AP^{-1}$ of Proposition \ref{prop:adjacency}.

(2) This follows from~(1) and Corollary~\ref{cor:IBN-graphs}.

(3) See Example~\ref{ex:type-changes}.

(4) Let $E$ be a graph with adjacency matrix $A$, and let $\sigma_0$ be the permutation of $E^0$
with matrix $P$. The equality $PAP^{-1}=A$ says that $|E^1(v,w)|=|E^1(\sigma_0(v),\sigma_0(w))|$
for all $v,w$. Choosing a bijection $E^1(v,w)\to E^1(\sigma_0(v),\sigma_0(w))$ for each pair
$(v,w)$ assembles into a bijection $\sigma_1$ of $E^1$ with $s\sigma_1=\sigma_0 s$ and
$r\sigma_1=\sigma_0 r$.
\end{proof}

Let $E$ be a graph, $H$ a subset of $E^0$, and $\sigma$ an automorphism of $E$. We say that $H$ is {\it hereditary} if, for every $e\in E^1$, $s(e) \in H$ implies $r(e)\in H$. We say that $H$ is {\it saturated} if, for every regular vertex $v$, the inclusion $r(s^{-1}(v))\subseteq H$ implies $v\in H$. We denote by $\mathcal{H}_E$ the set of saturated hereditary subsets of $E^0$. Set $$\mathcal{H}_E^{\sigma} = \{H\in \mathcal{H}_E \mid \sigma_0(H) = H\}.$$
Let $K$ be a field. We denote by $\mathcal{L}_{gr}(L_K(E))$ the lattice of graded ideals of the Leavitt path algebra $L_K(E)$. By Lemma \ref{lem:graphaut}, the graph automorphism $\sigma$ induces a graded algebra automorphism of $L_K(E)$ also denoted by $\sigma$. Set $$\mathcal{L}_{gr}^{\sigma}(L_K(E)) = \{I\in \mathcal{L}_{gr}(L_K(E))\mid \sigma(I) = I\}.$$
By Lemma \ref{lem:symmetry}, the graph automorphism $\sigma^{-1}$ induces a graded algebra automorphism of $L_K(E)^{\sigma}$ also denoted by $\sigma^{-1}$. Set 
$$\mathcal{L}_{gr}^{\sigma^{-1}}(L_K(E)^{\sigma}) = \{I\in \mathcal{L}_{gr}(L_K(E)^{\sigma})\mid \sigma^{-1}(I) = I\}.$$

\begin{proposition}\label{prop:gr-ideal}
Let $K$ be a field, $E$ a row-finite graph, and $\sigma$ an automorphism of $E$. Then    $$\mathcal{H}_E^{\sigma} \cong \mathcal{L}_{gr}^{\sigma}(L_K(E))\cong \mathcal{L}_{gr}^{\sigma^{-1}}(L_K(E)^{\sigma})$$ 
as lattices.
\end{proposition}
\begin{proof}
By Lemma~\ref{lem:path-endpoints}, there is a
path of length $k$ from $v$ to $w$ in $E_{\sigma}$ if and only if there is a path of length
$k$ from $v$ to $\sigma_0^{\,k}(w)$ in $E$. Consequently a $\sigma_0$-invariant subset
$H\subseteq E^0$ is hereditary (resp.\ saturated) in $E$ if and only if it is hereditary
(resp.\ saturated) in $E_{\sigma}$. Furthermore, if $I\in\mathcal{L}_{gr}^{\sigma}(L_K(E))$, then $H_I:=I\cap E^0$ is a saturated hereditary subset of $E^0$ satisfying $\sigma_0(H_I)=H_I$. Conversely, if $H\in\mathcal{H}_E^{\sigma}$, then the ideal $I(H)$ of $L_K(E)$ generated by $H$ belongs to $\mathcal{L}_{gr}^{\sigma}(L_K(E))$. Thus, using Theorem~\ref{thm:twist-is-LPA} together with \cite[Theorem 2.5.9]{AAS}, the proposition follows immediately.
\end{proof}

The following example shows that simplicity, primeness and the monoid $\mathcal V$ all
fail to be invariant under these twists.

\begin{example}\label{ex:R2R2}
Let $E=R_2\sqcup R_2$, 

\[\begin{tikzcd}
	{\bullet v_1} && {\bullet v_2}
	\arrow["{f_1}", from=1-1, to=1-1, loop, in=55, out=125, distance=10mm]
	\arrow["{f_2}", from=1-1, to=1-1, loop, in=100, out=170, distance=10mm]
	\arrow["{g_1}", from=1-3, to=1-3, loop, in=55, out=125, distance=10mm]
	\arrow["{g_2}", from=1-3, to=1-3, loop, in=100, out=170, distance=10mm]
\end{tikzcd}\]

Let $\sigma=(\sigma_0, \sigma_1)$ interchange the two components: $\sigma_0(v_1)=v_2$, $\sigma_0(v_2)=v_1$ and
$\sigma_1(f_i)=g_i$, $\sigma_1(g_i)=f_i$. Then the twisted graph $E_\sigma$ is as follows: 

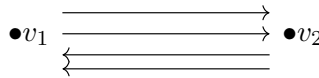
\[\begin{tikzcd}
	{\bullet v_1} &&& {\bullet v_2}
	\arrow[from=1-1, to=1-4]
	\arrow[shift left=3, from=1-1, to=1-4]
	\arrow[shift left=3, from=1-4, to=1-1]
	\arrow[shift left=5, from=1-4, to=1-1]
\end{tikzcd}\]

 Consequently,
\begin{enumerate}
\item[(i)] $L_K(E)\cong L_K(1,2)\times L_K(1,2)$ is neither simple nor prime, while
$L_K(E_\sigma)$ is purely infinite simple: $E_\sigma$ is cofinal, every vertex lies on a
cycle, and each cycle has an exit, while neither $\{v_1\}$ nor $\{v_2\}$ is hereditary.
\item[(ii)] $M_E=\{0,v_1,v_2,v_1+v_2\}$ is a semilattice with $v_i=2v_i$, whereas
$M_{E_\sigma}$ is generated by $v_1$ subject to $v_1=2v_2$, $v_2=2v_1$, hence
$v_1=4v_1$ and $M_{E_\sigma}\cong\{0\}\sqcup\mathbb Z/3\mathbb Z$. So
$\mathcal V\bigl(L_K(E)\bigr)\not\cong\mathcal V\bigl(L_K(E_\sigma)\bigr)$.
\item[(iii)] Nonetheless $\mathbf 1=v_1+v_2$ satisfies $2\cdot\mathbf 1=\mathbf 1$ in both
monoids, so both algebras have module type $(1,2)$.
\end{enumerate}
\end{example}

\begin{example}\label{ex:type-changes}
Let $E$ be the following graph 
\[\begin{tikzcd}
	{\bullet v} && {\bullet w}
	\arrow["{f_1}", from=1-1, to=1-1, loop, in=55, out=125, distance=10mm]
	\arrow["{f_2}", from=1-1, to=1-1, loop, in=100, out=170, distance=10mm]
	\arrow["h"{description}, from=1-1, to=1-3]
	\arrow["k"{description}, shift left=3, from=1-3, to=1-1]
	\arrow["{g_1}", from=1-3, to=1-3, loop, in=55, out=125, distance=10mm]
	\arrow["{g_2}", from=1-3, to=1-3, loop, in=100, out=170, distance=10mm]
\end{tikzcd}\]

Let $\sigma$ be the automorphism of $E$ given by
\[
\sigma_0(v)=w,\quad \sigma_0(w)=v,\quad \sigma_1(f_i)=g_i,\quad \sigma_1(g_i)=f_i,\quad
\sigma_1(h)=k,\quad \sigma_1(k)=h.
\]
Then $r_\sigma=\sigma_0^{-1}\circ r$ gives $r_\sigma(h)=v$, $r_\sigma(k)=w$,
$r_\sigma(f_i)=w$ and $r_\sigma(g_i)=v$. Thus $E_\sigma$ has one loop at each vertex, two edges
from $v$ to $w$, and two edges from $w$ to $v$.

\[\begin{tikzcd}
	{\bullet v} && {\bullet w}
	\arrow[from=1-1, to=1-1, loop, in=55, out=125, distance=10mm]
	\arrow[from=1-1, to=1-3]
	\arrow[shift left=3, from=1-1, to=1-3]
	\arrow[shift left=3, from=1-3, to=1-1]
	\arrow[shift left=5, from=1-3, to=1-1]
	\arrow[from=1-3, to=1-3, loop, in=55, out=125, distance=10mm]
\end{tikzcd}\]

\[
A_E=\begin{pmatrix}2&1\\1&2\end{pmatrix},\qquad
A_{E_\sigma}=A_EP^{-1}=\begin{pmatrix}1&2\\2&1\end{pmatrix}.
\]
Let $u=v+w$ denote the class of $1$.

In $M_E$, the relations are $v=2v+w$ and $w=v+2w$, that is, $v=v+u$ and $w=w+u$. Hence
$u=(v+u)+w=2u$, so $\LE$ has module type $(1,2)$.

In $M_{E_\sigma}$, the relations are $v=v+2w$ and $w=2v+w$. Adding them gives $u=3u$. On the
other hand, the assignment $v\mapsto 1$, $w\mapsto 0$ respects both relations and so defines a
monoid homomorphism $M_{E_\sigma}\to\Z/2\Z$; it sends $u\mapsto 1$ and $2u\mapsto 0$, so
$u\neq 2u$. Therefore, by Theorem~2.9, $\LE^\sigma\cong L_K(E_{\sigma})$ has module type
$(1,3)$.

Thus the module type is not invariant under the twist, in contrast with the graded module type
(Proposition \ref{prop:gr-IBN}). This is consistent with Corollary~\ref{cor:IBN-graphs}, since neither algebra
has IBN. Compared with Example \ref{ex:R2R2}, where the monoid $\mathcal V$ changes but the module type
survives, here the module type itself changes.
\end{example}

\begin{remark}\label{rem:graded-simplicity}
Graded simplicity is likewise not preserved. In Example \ref{ex:C2} the only hereditary
saturated subsets of $C_2$ are $\emptyset$ and $C_2^0$, so $L_K(C_2)$ is graded simple,
whereas $(C_2)_\sigma=R_1\sqcup R_1$ has four hereditary saturated subsets and
$L_K(C_2)^{\sigma}\cong K[x,x^{-1}]^2$ has four graded ideals. 
\end{remark}

\begin{remark}
We summarize the behavior of the standard invariants of a finite graph $E$ under the Zhang twist
$E\rightsquigarrow E_\sigma$.
\begin{center}
\begin{tabular}{@{}ll@{}}
\toprule
sinks, regular vertices, out-degrees & preserved (same source map)\\
acyclicity, Conditions (L) and (K) & preserved (Corollary \ref{cor:acyclic})\\
strong gradedness, von Neumann regularity & preserved (Proposition \ref{prop:strongly-graded-preserved})\\
Dedekind finiteness, stable finiteness, exchange property & preserved (Proposition \ref{prop:strongly-graded-preserved})\\
UGN property & preserved (Proposition \ref{prop:UGN}, Theorem~\ref{thm:IBN-finite-order})\\
IBN property & preserved (Corollary~\ref{cor:IBN-graphs})\\
gr-IBN and graded module type & preserved (Proposition \ref{prop:gr-IBN})\\
module type & not preserved (Example~\ref{ex:type-changes})\\
graded ideal lattice, graded simplicity & not preserved (Remark \ref{rem:graded-simplicity})\\
simplicity, primeness, $\mathcal V(\LE)$ & not preserved (Example \ref{ex:R2R2})\\
\bottomrule
\end{tabular}
\end{center}
\end{remark}

\bigskip

\noindent {\bf Acknowledgements.}
The authors would like to thank Xingting Wang for his helpful comments on an earlier draft of the paper. AI assistance was used in streamlining some of the arguments in Example \ref{ex:twist-not-LPA}, but the authors have verified it and they take all responsibility for it. The first author was supported by the Vietnam Academy of Science and Technology under grant CBCLCA.01/26-28, and by the Vietnam National Foundation for Science and Technology Development (NAFOSTED). Part of this work was done when the second author was giving a course on Leavitt path algebras in the CIMPA School at SASTRA University in India. He would like to thank the organizers Bernhard Keller and U. Arunachalam for inviting him to give the course in CIMPA School.


\end{document}